\documentclass[12pt]{article}

\usepackage{amsmath,amsfonts,yhmath,geometry,hyperref}
\usepackage{amsthm}
\usepackage{amssymb}
\usepackage[latin9]{inputenc}
\usepackage[nottoc,notlof,notlot]{tocbibind}
\usepackage{graphicx,float,psfrag}
\usepackage{mathtools}
\usepackage{floatflt}
\usepackage{ulem}
\usepackage{enumitem}
\usepackage{tikz-cd}
\usepackage{multirow}

\usepackage{pgf,tikz}

\usepackage{mathrsfs}
\usetikzlibrary{arrows}
\usetikzlibrary[patterns]

\definecolor{qqqqff}{rgb}{0.,0.,1.}
\definecolor{cqcqcq}{rgb}{0.7529411764705882,0.7529411764705882,0.7529411764705882}
\definecolor{ttqqqq}{rgb}{0.2,0.,0.}
\definecolor{qqqqff}{rgb}{0.,0.,1.}
\definecolor{xdxdff}{rgb}{0.49019607843137253,0.49019607843137253,1.}
\definecolor{zzttqq}{rgb}{0.6,0.2,0.}
\definecolor{cqcqcq}{rgb}{0.7529411764705882,0.7529411764705882,0.7529411764705882}

\definecolor{yqyqyq}{rgb}{0.5019607843137255,0.5019607843137255,0.5019607843137255}
\definecolor{uuuuuu}{rgb}{0.26666666666666666,0.26666666666666666,0.26666666666666666}
\definecolor{xdxdff}{rgb}{0.49019607843137253,0.49019607843137253,1.}
\definecolor{qqqqff}{rgb}{0.,0.,1.}

 \font\ncsc=cmcsc10
 \font\ntt=cmtt12

\usepackage{fancybox}

\newcommand{\PP}{\mathbb{P}}
\newcommand{\NN}{\mathbb{N}}
\newcommand{\ZZ}{\mathbb{Z}}
\newcommand{\RR}{\mathbb{R}}
\newcommand{\CC}{\mathbb{C}}

\newcommand{\dd}{ \text{d} }

\newcommand{\val}{\text{val}}

\newcommand{\fix}{\text{Fix}(\sigma)}

\newcommand{\tg}{ {\mathfrak{t}(\gamma)} }
\newcommand{\hg}{ {\mathfrak{h}(\gamma)} }
\newcommand{\tgs}{ {\mathfrak{t}(\gamma^\sigma)} }
\newcommand{\hgs}{ {\mathfrak{h}(\gamma^\sigma)} }
\newcommand{\gs}{ {\gamma^\sigma} }

\newcommand{\re}{\mathfrak{Re}}
\newcommand{\im}{\mathfrak{Im}}

\newcommand{\parfrac}[2]{\frac{\partial #1}{\partial #2}}

\newcommand{\q}[1]{q^{\frac{#1}{2}}-q^{-\frac{#1}{2}}}
\newcommand{\qd}{\q{1}}

\newtheorem{theo}{Theorem}[section]

\newtheorem{prop}[theo]{Proposition}
\newtheorem{lem}[theo]{Lemma}

\theoremstyle{definition}
\newtheorem{defi}[theo]{Definition}

\theoremstyle{remark}
\newtheorem{remark}[theo]{Remark}
\newenvironment{rem}[1]{
    \begin{remark}#1}{
    \xqed{\blacklozenge}\end{remark}
}

\theoremstyle{remark}
\newtheorem{example}[theo]{Example}
\newenvironment{expl}[1]{
    \begin{example}#1}{
    \xqed{\lozenge}\end{example}
}

\newcommand{\xqed}[1]{
    \leavevmode\unskip\penalty9999 \hbox{}\nobreak\hfill
    \quad\hbox{\ensuremath{#1}}}

\usepackage[off]{auto-pst-pdf}

\begin{document}
 
 
\title{A Tropical Computation of Refined Toric Invariants}
\author{Thomas Blomme}
\maketitle

\begin{abstract}
In \cite{mikhalkin2017quantum}, G. Mikhalkin introduced a refined count for the real rational curves in a toric surface which pass through certain conjugation invariant set of points on the toric boundary of the surface. Such a set consists of real points and pairs of complex conjugated points. He then proved that the result of this refined count depends only on the number of pairs of complex conjugated points on each toric divisor. Using the tropical geometry approach and the correspondence theorem, he managed to compute these invariants if all the points of the configuration are real. In this paper we address the case when the configuration contains some pairs of conjugated points, all belonging to the same component of the toric boundary.
\end{abstract}

\tableofcontents

\section{Introduction}

\subsection{Curves in a toric surface and their enumeration} 

The paper deals with enumerative problems involving rational curves in toric surfaces. Let $N$ be a $2$-dimensional lattice with basis $(e_1,e_2)$, and let $M$ be the dual lattice, with the dual basis $(e_1^*,e_2^*)$. The lattices $M$ and $N$ are called, respectively, the lattices of characters and co-characters. Let $\Delta=(n_j)\subset N$ be a multiset of lattice vectors, whose total sum is zero. We denote by $m$ the cardinal of $\Delta$. Let $P_\Delta\subset M$ be the convex lattice polygon in $M$ defined up to translation, whose normal vectors oriented outside $P_\Delta$ are the vectors of $\Delta$, counted with multiplicity. This means that the lattice length of a side $E$ of $P_\Delta$ is precisely the sum of the lattice lengths of the vectors of $\Delta$ which are normal to the side $E$. The polygon $P_\Delta$ is well-defined up to translation. The number of lattice points on the boundary of $P_\Delta$ is thus equal to the sum of the lattice lengths of the vectors of $\Delta$. The vectors of $\Delta$ define a fan $\Sigma_\Delta$ in $N_\RR$, only depending on $P_\Delta$. We denote by $\CC\Delta$ the associated toric surface, whose dense complex torus is $\text{Hom}(M,\CC^*)\simeq N\otimes\CC^*$. The toric divisors  of $\CC\Delta$ are in bijection with the sides of the polygon $P_\Delta$. The complex conjugation on $\CC$ defines an involution on $\text{Hom}(M,\CC^*)$, which extends to $\CC\Delta$. This additional structure makes it into a real surface. Its fixed locus, also called the real locus, is denoted by $\RR\Delta$.\\

\begin{expl}
\begin{itemize}[label=-]
\item Let $\Delta_d=\{-e_1^d,-e_2^d,(e_1+e_2)^d\}$, where the exponent $d$ means that the multiplicity in the multiset is $d$. The associated polygon is $P_d=\text{Conv}( 0,d\cdot e_1^*,d\cdot e_2^* )$, the standard triangle of size $d$. The associated toric surface is the projective plane $\CC P^2$, and the real locus is the real projective plane $\RR P^2$.
\item If we set $\square_{a,b}=\{-e_1^b,e_1^b,-e_2^a,e_2^b\}$, the polygon $P_{a,b}$ is the rectangle $\text{Conv}( 0,a\cdot e_1^*,b\cdot e_2^*,a\cdot e_1^*+b\cdot e_2^* )$, and the associated toric surface is $\CC P^1\times\CC P^1$, whose real locus is $\RR P^1\times\RR P^1$.
\end{itemize}
\end{expl}
From now on, and unless stated otherwise, $\Delta\subset N$ is a family of primitive vectors, referred as \textit{toric degree}. A rational curve $\CC P^1\rightarrow\CC \Delta$ of degree $\Delta$ is a curve that admits a parametrization
$$\varphi :t\in\CC \dashrightarrow \chi\prod_{i=1}^m(t-\alpha_i)^{n_i}\in\text{Hom}(M,\CC^*),$$
where the numbers $\alpha_i$ are some complex points inside $\CC$, and $\chi:M\rightarrow\CC^*$ is a complex co-character, \textit{i.e.} an element of $N\otimes\CC^*$. The vectors $n_i$ can also be recovered as the functions $m\mapsto\text{val}_{\alpha_i}(\chi^m)$, where $\chi^m$ denotes the monomial function associated to the character $m\in M$, and $\text{val}_{\alpha_i}$ is the order of vanishing at $\alpha_i$. Such a parametrization is unique up to automorphism of $\CC P^1$, therefore the space of rational curves $\CC P^1\rightarrow\CC\Delta$ is of dimension $m-1$. A rational curve is real if it admits a conjugation invariant parametrization. Equivalently, $\chi$ should have real values, and the numbers $\alpha_i$ are either real, or come in pairs of conjugated points with the same exponent vector $n_i$.\\

In this setting, we choose a generic configuration $\mathcal{P}$ of $m-1$ points inside $\CC\Delta$, and look for rational curves of degree $\Delta$ passing through this configuration. We have a finite number of complex solutions. If we denote by $\mathcal{S}^\CC(\mathcal{P})$ the set of solutions, it is a classical result that the cardinal $|\mathcal{S}^\CC(\mathcal{P})|$ is independent of the point configuration $\mathcal{P}$. Its value is denoted by $N_\Delta$.\\

Over the real field, the situation is different. If we choose a conjugation invariant configuration of points $\mathcal{P}$, called a \textit{real configuration}, meaning that it consists of real points and pairs of conjugated points, and denote by $\mathcal{S}^\RR(\mathcal{P})$ the set of real rational curves passing through $\mathcal{P}$, the value of the cardinal $|\mathcal{S}^\RR(\mathcal{P})|$ depends on the choice of $\mathcal{P}$. However, J-Y. Welschinger \cite{welschinger2005invariants} showed that for del Pezzo surfaces, if the curves are counted with an appropriate sign, the count of real solutions depends only on the number $s$ of pairs of conjugated points in the configuration, yielding an invariant denoted by $W_{\Delta,s}$.\\

While the values of $N_\Delta$ were already known, those of $W_{\Delta,s}$ were computed roughly at the same time their invariance was proven. In \cite{mikhalkin2005enumerative}, G. Mikhalkin proved a correspondence theorem along with a lattice path algorithm that provided a way of computing both invariants $N_\Delta$ and $W_{\Delta,0}$ (only real points in the configuration) using the tropical geometry approach. Later E. Shustin \cite{shustin2002patchworking,shustin2004tropical} also used tropical geometry to compute the Welschinger invariants $W_{\Delta,s}$ for any value of $s$. Both approaches use the fact that curves and hypersurfaces coincide inside surfaces. These approaches do not easily generalize to curve count in any dimension. Since, the correspondence theorem has first been generalized for rational curves in any dimension by T. Nishinou and B. Siebert \cite{nishinou2006toric} using logarithmic stable maps. For other references regarding generalization to higher dimension, the reader can refer to G. Mikhalkin \cite{mikhalkin2006tropical}, and to I. Tyomkin \cite{tyomkin2012tropical}, whose proof largely inspires the proof of Theorem \ref{realization theorem}. Their approaches provide another proof of the previous correspondence theorem, in the case of rational curves in toric surfaces.\\

To compute the values of $N_\Delta$ and $W_{\Delta,0}$, Mikhalkin counts tropical curves solution to the analog tropical enumerative problem with two specific choices of integer multiplicity. Following his computation, F. Block and L. G\"ottsche \cite{block2016refined} proposed a way of combining these integer multiplicities by refining them into a Laurent polynomial multiplicity, which gives back the values of $N_\Delta$ and $W_{\Delta,0}$ when evaluated at $\pm 1$ respectively. This new refined multiplicity was proved in \cite{itenberg2013block} to give a tropical invariant. This choice of multiplicity seems to appear in a growing number of situations, while its meaning in classical geometry remains quite mysterious. Conjecturally, this invariant is suspected to coincide with the refinement of Severi degrees by the $\chi_{-y}$-genera proposed by L. G\"ottsche and V. Shende in \cite{gottsche2014refined}. This invariant bears also similarities with some Donaldson-Thomas wall-crossing invariants considered by M. Kontsevich and Y. Soibelman \cite{kontsevich2008stability}.\\

In \cite{mikhalkin2017quantum}, Mikhalkin proved that the refined multiplicity could also be used to count real rational curves passing through a fixed symmetric real configuration of points on the toric boundary of a surface, according to the value of a so-called \textit{quantum index}, count for which he proved the classical invariance. The details of these results are explained below.

\subsection{Quantum indices of real curves}

In \cite{mikhalkin2017quantum}, Mikhalkin introduced a quantum index for oriented real curves. We now recall the definition of the quantum index.\\

Let $\varphi:\CC C\rightarrow\CC\Delta$ be a real algebraic curve, with $\CC C$ a smooth Riemann surface, and $\RR C$ the real locus. 
We say that this curve is of \textit{type $I$} if $\CC C\backslash\RR C$ is disconnected. A choice of a connected component $S$  of $\CC C\backslash\RR C$ induces an orientation of its boundary $\RR C$, this orientation is 
called a \textit{complex orientation}. The choice of the other component leads to the opposite orientation. When $S$ is given, we denote by $\overrightarrow{C}$ the curve endowed with the complex orientation induced by $S$.\\ 

The map $n\otimes z\mapsto 2 n\otimes z=n\otimes z^2$ from $N\otimes\CC^*$ to itself extends to a map $\text{Sq}$ from $\CC\Delta$ to itself, called the square map. In coordinates it is just the map that squares each coordinate. We say that $\CC C$ has real or purely imaginary intersection points with the toric divisors if the images of the intersection points with the toric boundary under the square map are real, \textit{i.e.} belong to $\RR\Delta$. We have the logarithmic map
$$\text{Log}:n\otimes z\in N\otimes\CC^*\longmapsto n\otimes\text{Log}|z|\in N_\RR=N\otimes\RR.$$
In a basis of $N$, it is the logarithm of the absolute value coordinate by coordinate. Let $\omega$ be a generator of $\Lambda^2 M$, \textit{i.e.} a non-degenerated $2$-form on $N$, that extends to a volume form on $N_\RR$. By pulling back the volume form to $N\otimes\CC^*$, we can define the logarithmic area of $S$:
$$\mathcal{A}_{\text{Log}}(S)=\int_{\varphi(S)}\text{Log}^*\omega.$$
This area depends on the complex orientation up to a change of sign.

\begin{theo}[Mikhalkin\cite{mikhalkin2017quantum}]
Let $\varphi : \CC C\rightarrow\CC\Delta$ be a real type I curve with real or purely imaginary intersection points with the toric divisors. Let $S$ be one of the two connected components of $\CC C\backslash\RR C$. Then, there exists a half-integer $k(S,\varphi)$, called the quantum index of $S$, or quantum index of $\overrightarrow{C}$, such that 
$$\mathcal{A}_{\text{Log}}(S)=\int_{\varphi(S)}\text{Log}^*\omega=k(S,\varphi)\pi^2.$$
\end{theo}

This theorem concerns all real type I curves but will only be of interest to us in the case of rational curves, which are automatically of type I as long as their real part is non-empty. So, they have a well-defined quantum index provided that the non-real intersection points of $\varphi(\CC C)$ with the toric boundary have a purely imaginary coordinate.\\ 

Let $C$ be a real type $I$ curve, and let $S$ be a connected component of the complement of the real part of $C$. The area of the coamoeba of the curve is defined as follows: if $2\arg : \varphi(S)\subset N\otimes\CC^*\rightarrow  N\otimes\RR/\pi\ZZ$ is twice the argument map, and $\omega_\theta$ denotes the volume form on $N\otimes\RR/\pi\ZZ$ induced by $\omega$, then
$$\mathcal{A}_{\arg}=\int_{\varphi(S)} (2\arg)^*\omega_\theta.$$
This area coincides with $\mathcal{A}_\text{Log}$. Thus, it is equal to the quantum index whenever it is defined. Moreover, Mikhalkin provided a way to systematically compute the quantum index in the case the curve is of \textit{toric type $I$}. For rational curves that means that all the intersection points with the toric boundary are real. In order to compute the quantum index in the case 
the curve has non-real intersection points, we slightly enlarge the computation approach.\\

Finally, Mikhalkin proved that due to the additive character of the area, the quantum index is also additive. Therefore, if we consider a family of curves that admits a tropical limit, then close to the tropical limit, provided that the real oriented curves in the family have a well-defined quantum index, the quantum indices of the curves are constant, equal to the sum of the quantum indices of the curves associated to the vertices of the resulting tropical curve. This statement is made precise later, otherwise, see \cite{mikhalkin2017quantum}.\\


Using this approach, the computation of the quantum indices of curves close to the tropical limit is reduced to a few local cases. In \cite{mikhalkin2017quantum}, the only needed local case is the one of a rational curve with three real boundary points. In this paper, we also have to consider the additional case of a rational curve with two real and two complex conjugated intersection points with the toric boundary.

\subsection{Refined enumerative geometry} 
Now, let $\mathcal{P}_0$ be a real configuration of $m$ points taken on the toric boundary of $\CC\Delta$, and such that each irreducible component of the toric boundary contains a number of points equal to the integral length of the corresponding side of the polygon $\Delta$, \textit{e.g.} for the projective plane, there are $d$ points per axis. Let $r$ be the number of real points in $\mathcal{P}_0$, and $s$ the number of pairs of complex conjugated points in $\mathcal{P}_0$, so that $r+2s=m$. If the polygon $\Delta$ has $p$ sides, let $\lambda_i$ be the number of pairs of conjugated points on the side $i$, so that $\sum_1^p\lambda_i=s$. We moreover assume that the pairs of complex conjugated points are purely imaginary, and that there is at least one real point in the configuration. The Vi\`ete formula ensures that there exists a curve not containing the boundary as a component and passing through the configuration $\mathcal{P}_0$ only if the configuration is subject to the \textit{Menelaus condition}, which we therefore assume: the product of the coordinates of the points is equal to $(-1)^m$. We finally enlarge $\mathcal{P}_0$ into $\mathcal{P}$ by adding the opposite of the real points. The opposite of the non-real
points, being their conjugated since they are purely imaginary, are already there. We obtain a symmetric configuration of points which is invariant by conjugation.\\

Let $\mathcal{S}(\mathcal{P})$ be the set of oriented real rational curves of degree $\Delta$ that contain either $p_i$ or $-p_i$ for every $p_i\in\mathcal{P}_0$. It means that the curve passes through one point of each pair of opposite real points, and through both points of each pair of complex conjugated points since if it passes through one, it also passes through its conjugate. This set of curves could also be defined in the following way: if $\text{Sq}:\CC\Delta\rightarrow\CC\Delta$ denotes the square map, then $\mathcal{S}(\mathcal{P})$ is the set of oriented real rational curves $C$ of degree $\Delta$ such that $\text{Sq}(\RR C)$ contains $\text{Sq}(\mathcal{P}_0)$, which is a configuration of real points, since the square of a purely imaginary point is a negative real one. All the oriented curves of $\mathcal{S}(\mathcal{P})$ have a well-defined quantum index. For an oriented curve $\overrightarrow{C}\in\mathcal{S}(\mathcal{P})$, let $\psi:\CC P^1\rightarrow \CC C$ be a parametrization of $C$, with orientation given by $S\subset\CC P^1\backslash\RR P^1$. We consider the composition of the parametrization $\psi$ with the logarithm map, restricted to the real locus:
$$\text{Log}|\psi|:\RR P^1\rightarrow N\otimes\RR=N_\RR.$$
Its image is $\text{Log}(\psi(\RR P^1))\subset N_\RR$. We now consider the logarithmic Gauss map that associates to each point of $\RR P^1$ the tangent direction to $\text{Log}(\psi(\RR P^1))$, which is a point in $\PP^1(N_\RR)\simeq\RR P^1$, oriented by $\omega$. Since both copies of $\RR P^1$ are oriented, one can consider the degree $\text{Rot}_\text{Log}(\overrightarrow{C})$ of the logarithmic Gauss map, \textit{i.e. }when the domain $\RR P^1$ is oriented by the choice of the complex orientation of $C$, and the target $\mathbb{P}^1(N_\RR)\simeq\RR P^1$ oriented by $\omega$. Let $\sigma(\overrightarrow{C})=(-1)^\frac{m-\text{Rot}_\text{Log}(\overrightarrow{C})}{2}$, which is equal to $\pm 1$. We then set
$$R_{\Delta,k}(\mathcal{P})=\sum_{C\in\mathcal{S}_k(\mathcal{P})} \sigma(\overrightarrow{C})\in\ZZ,$$
where $\mathcal{S}_k(\mathcal{P})$ denotes the subset of $\mathcal{S}(\mathcal{P})$ of oriented curves having quantum index $k$. Finally, we define
$$R_\Delta(\mathcal{P})=\frac{1}{4}\sum_{k}R_{\Delta,k}(\mathcal{P})q^k\in \ZZ\left[ q^{\pm\frac{1}{2}}\right].$$

\begin{theo}[Mikhalkin\cite{mikhalkin2017quantum}]
As long as the configuration $\mathcal{P}$ is generic, the Laurent polynomial $R_{\Delta}(\mathcal{P})$ only depends on $\lambda=(\lambda_1,\dots,\lambda_p)$.
\end{theo}

The above Laurent polynomial only depending on $\lambda$ is denoted by $R_{\Delta,\lambda}$.

\begin{rem}
It is important for the result and for its proof to consider not only curves passing through $\mathcal{P}_0$ but also through the symmetric configuration $\mathcal{P}$, otherwise the invariance might fail.
\end{rem}

In the case of a totally real configuration of points, \textit{i.e.} $\lambda_i=0$, using the correspondence theorem, Mikhalkin proved that the invariant $R_{\Delta,(0)}$ coincides, up to a normalization, with the tropical refined invariant $N_\Delta^{\partial,\text{trop}}$. This invariant is obtained as follows. (For a more complete description, see section \ref{tropical enumerative inv}.) We consider tropical curves of degree $\Delta\subset N$. For each vector $n_j$ in $\Delta$, one chooses an oriented line directed and oriented by $n_j$, and such that the configuration of lines is generic. Then we count rational tropical curves of degree $\Delta$ whose unbounded ends are contained in the chosen lines, using the Block-G\"ottsche multiplicities from \cite{block2016refined}. The result does not depend on the chosen configuration, and is denoted by $N_\Delta^{\partial,\text{trop}}$.

\begin{theo}[Mikhalkin\cite{mikhalkin2017quantum}]
One has
$$R_{\Delta,(0)}=(q^\frac{1}{2}-q^{-\frac{1}{2}})^{m-2}N_\Delta^{\partial,\text{trop}}.$$
\end{theo}

\begin{rem}
The exponent $m-2$ corresponds to the number of vertices of a tropical curve involved in the enumerative problem defining $N_\Delta^{\partial\text{trop}}$. Thus, the normalization $(q^\frac{1}{2}-q^{-\frac{1}{2}})^{m-2}$ amounts to clear the denominators of the Block-G\"ottsche multiplicities from \cite{block2016refined}. We recall their definition in section \ref{tropical enumerative inv}.
\end{rem} 

The results of Mikhalkin reduce the computation of the invariants $R_{\Delta,(0)}$ to a tropical computation of $N_\Delta^{\partial,\text{trop}}$. In this paper we prove that the computation of the $R_{\Delta,(\lambda_1,0,\dots,0)}$, \textit{i.e.} in the case of a unique boundary component carrying complex conjugated points, can also be carried out using tropical geometry.\\

More precisely, we have the following theorem. Let $\Delta=(n_j)\subset N$ still be a family of primitive vectors whose total sum is zero, allowing us to consider curves of degree $\Delta$ inside the toric surface $\CC\Delta$. We fix a side $E_1$ of $P_\Delta$, normal to a vector $n_1\in\Delta$. Let $l$ be the lattice length of $E_1$, which is the multiplicity of $n_1$ in $\Delta$. For $s\leqslant\frac{l}{2}$, we denote by $\Delta(s)$ the family $(\Delta\backslash\{n_1^{2s}\})\cup\{(2n_1)^s\}$, where $2s$ copies of $n_1$ are replaced by $s$ copies of $2n_1$. The degree $\Delta(s)$, not consisting only of primitive vectors anymore, allows us to consider tropical curves of degree $\Delta(s)$, and the associated refined invariant $N_{\Delta(s)}^{\partial,\text{trop}}$.

\begin{theo}
\label{theorem paper}
One has
$$R_{\Delta,(s,0,\dots,0)}=\frac{(\qd)^{m-2-s_1}}{(q-q^{-1})^{s_1}}N^{\partial,\text{trop}}_{\Delta(s)} =\frac{(\qd)^{m-2-2s}}{(q^\frac{1}{2}+q^{-\frac{1}{2}})^{s}}N^{\partial,\text{trop}}_{\Delta(s)}  $$
\end{theo}

The paper is organized as follows. In the second section we recall the standard definitions related to tropical curves and the tropicalization. In the third section we recall the definition of the quantum index and its computation in few cases. The fourth section is devoted to the enumerative problems leading to the definition of $R_{\Delta,(s,0,\dots,0)}$ and $N_{\Delta(s)}^{\partial,\text{trop}}$. In the fifth section, we prove the correspondence theorem before proving Theorem \ref{theorem paper} in the last section.

\section{Tropical curves  and real tropical curves}

	\subsection{Real abstract tropical curves}

Let $\overline{\Gamma}$ be a finite connected graph without bivalent vertices. Let $\overline{\Gamma}_\infty^0$ be the set of $1$-valent vertices of $\overline{\Gamma}$, and $\Gamma=\overline{\Gamma}\backslash\overline{\Gamma}^0_\infty$. If $m$ denotes the cardinal of $\overline{\Gamma}_\infty^0$, its elements are labeled with integers from $[\![1;m]\!]$. We denote by $\Gamma^0$ the set of vertices of $\Gamma$, and by $\Gamma^1$ the set of edges of $\Gamma$. The non-compact edges resulting from the eviction of $1$-valent vertices are called \textit{unbounded ends}. The set of unbounded ends is denoted by $\Gamma^1_\infty$, while its complement, the set of bounded edges, is denoted by $\Gamma_b^1$. Notice that $\Gamma_\infty^1$ is also labeled by $[\![1;m]\!]$. Let $l:\gamma\in\Gamma_b^1\mapsto |\gamma|\in\RR_+^*=]0;+\infty[$ be a function, called length function. It endows $\Gamma$ with the structure of a metric graph by decreting that a bounded edge $\gamma$ is isometric to $[0;|\gamma|]$, and an unbounded end is isometric to $[0;+\infty[$.

\begin{defi}
Such a metric graph $\Gamma$ is called an \textit{abstract tropical curve}.
\end{defi}

An isomorphism between two abstract tropical curves $\Gamma$ and $\Gamma'$ is an isometry $\Gamma\rightarrow\Gamma'$. In particular an automorphism of $\Gamma$ does not necessarily respect the labeling of the unbounded ends since it only respects the metric. Therefore, an automorphism of $\Gamma$ induces a permutation of the set $I=[\![1;m]\!]$ of unbounded ends.

\begin{defi}
Let $\Gamma$ be an abstract tropical curve. A \textit{real structure} on $\Gamma$ is an involutive isometry $\sigma:\Gamma\rightarrow\Gamma$. A \textit{real abstract tropical curve} is an abstract tropical curve enhanced with a real structure.
\end{defi}

Since a real structure $\sigma:\Gamma\rightarrow\Gamma$ has to preserve the metric, for any bounded edge $\gamma$, one has $|\gamma|=|\sigma(\gamma)|$. The real structure also induces an involution on the set of ends $I=[\![1;m]\!]$ of $\Gamma$. The fixed ends are called \textit{real ends} and the pairs of exchanged ends are called the \textit{conjugated ends}, or \textit{complex ends}. The fixed locus of $\sigma$ is denoted by $\fix$. It is a subgraph of $\Gamma$.\\

\begin{expl}
\begin{itemize}[label=-]
\item The trivial real structure $\sigma=\text{id}_\Gamma$ is the most common example, useful despite its simplicity.
\item If $\Gamma$ is an abstract tropical curve and $e,e'\in\Gamma^1_\infty$ are two unbounded ends adjacent to the same vertex $w$, another example is given by permuting the two unbounded ends $e$ and $e'$, and leaving the rest of the graph invariant.
\end{itemize}
\end{expl}

	\subsection{Real parametrized tropical curves}

Recall that we have two dual lattices $N$ and $M$, and $N_\RR=N\otimes\RR$. We now define parametrized tropical curves in $N_\RR$.

\begin{defi}
A \textit{parametrized tropical curve} in $N_\RR\simeq\RR^2$ is a pair $(\Gamma,h)$, where $\Gamma$ is an abstract tropical curve and $h:\Gamma\rightarrow\RR^2$ is a map satisfying the following requirements:
\begin{itemize}
\item For every edge $E\in\Gamma^1$, the map $h|_E$ is affine. If we choose an orientation of $E$, the value of the differential of $h$ taken at any interior point of $E$, evaluated on a tangent vector of unit length, is called the slope of $h$ alongside $E$. This slope must lie in $N$.
\item We have the so called \textit{balancing condition}: at each vertex $V\in\Gamma^0$, if $E$ is an edge containing $V$, and $u_E$ is the slope of $h$ along $E$ when $E$ is oriented outside $V$, then 
$$\sum_{E:\partial E\ni V} u_E=0\in N.$$
\end{itemize}
\end{defi}

Two parametrized curves $h:\Gamma\rightarrow N_\RR$ and $h':\Gamma'\rightarrow N_\RR$ are isomorphic if there exists an isomorphism of abstract tropical curves $\varphi:\Gamma\rightarrow\Gamma'$ such that $h=h'\circ\varphi$.\\

\begin{defi}
A \textit{real parametrized tropical curve} is a triplet $(\Gamma,\sigma,h)$, where $(\Gamma,h)$ is a parametrized tropical curve, $\sigma$ is a real structure on $\Gamma$, and $h$ is $\sigma$-invariant: $h\circ\sigma=h$.
\end{defi}

\begin{rem}
In particular, two vertices that are exchanged by $\sigma$ have the same image under $h$, and two edges that are exchanged by $\sigma$ have the same slope and the same image. Such edges are called \textit{double edges}. If they are unbounded, we call them a \textit{double end}. Thus, the image $h(\Gamma)\subset N_\RR$ may not be sufficient to recover $\Gamma$ and the real structure, since for instance there is no way of distinguishing a double end from a simple end with twice their slope.
\end{rem}

\begin{figure}
\begin{center}

\begin{tikzpicture}[line cap=round,line join=round,>=triangle 45,x=0.5cm,y=0.5cm]
\clip(0.,0.) rectangle (12.,9.);
\draw (0.,9.)-- (2.,6.);
\draw (2.,6.)-- (0.,5.);
\draw (2.,6.)-- (4.,5.);
\draw (0.,8.8)-- (2.,5.8);
\draw (2.,5.8)-- (0.,4.8);
\draw (2.,5.8)-- (4.,4.8);

\draw (4.,5.)-- (3.,2.);
\draw (3.,2.)-- (0.,2.);
\draw (3.,2.)-- (4.,0.);
\draw (4.,5.)-- (7.,5.);

\draw (7.,5.)-- (9.,8.);
\draw (6.8,5.)-- (8.8,8.);

\draw (7.,5.)-- (9.,2.);
\draw (9.,2.)-- (8.,0.);
\draw (9.,2.)-- (12.,2.);
\end{tikzpicture}
\caption{Abstract real tropical curve with its real structure depicted by doubling the exchanged edges.}
\end{center}
\end{figure}
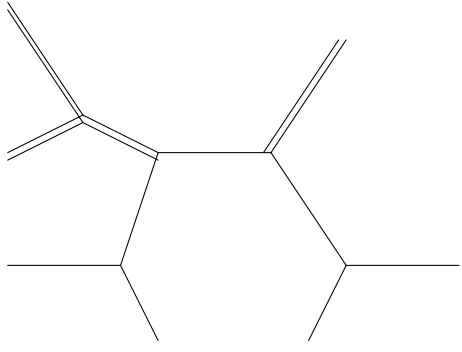

\begin{rem}
We could assume that $M=N=\ZZ^2$, but the distinction is now useful since the lattice $M$ is a set of functions on the space $N_\RR$ where the tropical curves live, while $N$ is the space of the slopes of the edges of a tropical curve. Moreover, notice that we deal with tropical curves in the affine space $N_\RR$, identified with its tangent space at $0$.
\end{rem}

If $e\in\Gamma^1_\infty$ is an unbounded end of $\Gamma$, let $n_e\in N$ be the slope of $h$ alongside $e$, oriented out of its unique adjacent vertex, \textit{i.e.} toward infinity. The multiset
$$\Delta=\{n_e\in N|e\in\Gamma^1_\infty \}\subset N,$$
is called the \textit{degree} of the parametrized curve. It is a multiset since an element may appear several times. Let $(e_1,e_2)$ be a basis of $N$. We say that a curve is of degree $d$ if $\Delta=\Delta_d=\{(-e_1)^d,(-e_2)^d,(e_1+e_2)^d\}$. Using the balancing condition, one can show that $\sum_{n_e\in\Delta}n_e=0$.\\


\begin{defi}
\begin{itemize}
\item[-] Let $\Gamma$ be an abstract tropical curve. The \textit{genus} of $\Gamma$ is its first Betti number $b_1(\Gamma)$.
\item[-] A curve is \textit{rational} if it is of genus $0$.
\item[-] A parametrized tropical curve $(\Gamma,h)$ is \textit{rational} if $\Gamma$ is rational.
\end{itemize}
\end{defi}

\begin{rem}
A parametrized tropical curve is then rational if the graph that parametrizes it is a tree.
\end{rem}

	\subsection{Plane tropical curves}
	
	We now look at plane tropical curves, which are the images of the parametrized tropical curves. To define a plane tropical curve, we consider a \textit{tropical polynomial}: for $x\in N_\RR$, put
	$$P(x)=\max_{u\in P_\Delta}(a_u+\langle u,x\rangle ),$$
	where $P_\Delta\subset M$ is a convex lattice polygon, and $a_u\in\RR\cup\{-\infty\}$ are the coefficients of the polynomial, different from $-\infty$ if $m\in M$ is a vertex of $P_\Delta$. The polygon $P_\Delta$ is called the \textit{Newton polygon} of the plane tropical curve. If we choose a basis of $M$ and $N$, the tropical polynomial $P$ takes the following form:
	$$P(x,y)=\max_{(i,j)\in P_\Delta}(a_{ij}+ix+jy).$$
	The tropical polynomial $P$ is then a piecewise affine convex function which is the maximum of a finite number of affine functions. We assume that $P_\Delta$ contains more than one point, otherwise $P$ is an affine function. The tropical polynomial $P$ induces a subdivision of $P_\Delta$ with the following rule: $u,u'\in P_\Delta$ are connected by an edge if $\{x\in N_\RR : P(x)=a_u+\langle u,x\rangle=a_{u'}+\langle u',x\rangle\}\neq\emptyset$. The \textit{corner locus} $C$ of $P$, \textit{i.e.} the set where at least two of the affine functions realize the maximum, is a rectilinear graph in $N_\RR$. Equivalently, this is the set of points where $P$ is not differentiable. The subdivision of $P_\Delta$ induced by $P$ is \textit{dual} to the corner locus $C$ in the sense that there exists natural bijections between the following pairs of sets: edges of $C$ and edges of the subdivision, vertices of $C$ and polygons of the subdivision, components where $P$ is smooth equal to one of the affine functions and vertices of the subdivision.
	
	\begin{defi}
The \textit{plane tropical curve} $C$ associated to a tropical polynomial $P$ is the corner locus of $P$, enhanced with the following weights on the edges of $C$: the weight of an edge is the lattice length of the dual edge in the subdivision of $P_\Delta$. The polygon $P_\Delta$ is called the \textit{degree} of the curve $C$.
\end{defi}

Plane tropical curves can be characterized as finite weighted graphs (weights on the edges) with unbounded ends in $N_\RR$, such that the edges are affine with slope in $N$, and the vertices satisfy the following balancing condition: if $E$ is an edge of weight $w_E$ adjacent to a vertex $V$, and $u_E$ is a primitive lattice vector directing $E$ oriented outward from $V$, we have
$$\sum_{E\ni V}w_Eu_E=0.$$
For more details on plane tropical curves, see \cite{brugalle2014bit}. \\

One can show that the image of a parametrized tropical curve is indeed a plane tropical curve with this definition. Moreover, the relation between $\Delta$, degree of the parametrized tropical curve, and $P_\Delta$, degree of the plane curve, is as described in the introduction. Conversely, a plane tropical curve can always be parametrized by an abstract tropical curve, leading to a parametrized tropical curve. Moreover, if $C$ is a plane curve parametrized by $h:\Gamma\rightarrow N_\RR$, the weight of an edge $E$ of $C$ can be recovered as the sum of the lattice lengths of the slopes of $h$ on the edges $\gamma\in\Gamma$ which project onto $E$. However, there are often many ways of choosing a parametrization of a plane tropical curve, by non-isomorphic tropical curves. In fact, even the degree $P_\Delta\subset M$ does not uniquely determine the degree $\Delta\subset N$ of a parametrized curve parametrizing $C$: \textit{e.g.} if $C$ has un unbounded end of weight $2$, a parametrizing graph $\Gamma$ could have either an end $e$ with $h|_e$ having slope of lattice length $2$, or two ends of primitive slope with the same image by $h$.\\

We now define the usual concepts associated to plane classical curves in the case of plane tropical curves, starting with reducible curves. 

\begin{defi}
\begin{itemize}[label=-]
\item A plane tropical curve is \textit{reducible} if it can be represented as the union of two distinct plane tropical curves. 
\item A plane tropical curve is \textit{irreducible} if it is not reducible.
\end{itemize}
\end{defi}

Going on with the definition of the genus, one needs to be careful since it depends on the chosen parametrization.

\begin{defi}
\begin{itemize}[label=-]
\item The \textit{genus} of a plane tropical curve is the smallest genus among its possible parametrizations.
\item A plane tropical curve is \textit{rational} if it is irreducible and can be parametrized by a rational tropical curve.
\end{itemize}
\end{defi}

One can show that if $C$ is an irreducible rational plane tropical curve with unbounded ends of weight $1$, it admits a unique rational parametrization. More generally, we have the following statement.

\begin{prop}\cite{mikhalkin2005enumerative}\label{unique parametrization}
Let $C$ be a rational plane tropical curve, and let $u_e$ be a directing primitive lattice vector for each unbounded end $e$, oriented toward infinity. Let $w_e$ be the weight of $e$. Then $C$ is the image of a unique rational parametrized tropical curve of degree $\Delta=\{w_e u_e\}_e$.
\end{prop}

\subsection{Real parametrizations of a plane tropical curve}
 
In this subsection, we extend Proposition \ref{unique parametrization} by describing the possible real rational parametrizations of an irreducible rational plane tropical curve, with unbounded ends of weights $1$ or $2$.\\

Let $C$ be a rational plane tropical curve with unbounded ends of weight $1$ or $2$. Let $u_1$, $\dots$, $u_r$, $2v_1$, $\dots$, $2v_s$ be the weighted directing vectors of the unbounded ends of $C$, with vectors $u_i,v_j$ being primitive vectors in $N$. We assume that $r\geqslant 1$. Let $h:\Gamma\rightarrow N_\RR$ be the unique rational parametrization of $C$ given by Proposition \ref{unique parametrization}, which is of degree $\{u_i,2v_j\}_{i,j}$. We now describe the parametrizations of $C$ by real parametrized rational curves of degree $\{u_i,v_j^2\}_{i,j}$, which means that now all vectors are primitive, and each unbounded end of weight $2$ is replaced with two ends of weight $1$.\\



We define the subgraph $\Gamma_\text{even}$ of $\Gamma$ as the minimal subgraph satisfying both following requirements:
	\begin{itemize}[label=-]
	\item Every unbounded end of $\Gamma$ of even weight (\textit{i.e.} mapped to an end of $C$ directed by $2v_j$ for some $j$) is in $\Gamma_\text{even}$.
	\item If $V$ is a vertex of $\Gamma$ and all edges adjacent to $V$ but one are in $\Gamma_\text{even}$, then the remaining adjacent edge also is in $\Gamma_\text{even}$. Following \cite{shustin2004tropical}, such a vertex is called an \textit{extendable vertex}.
	\end{itemize}

\begin{rem}
The subgraph $\Gamma_\text{even}$ is the maximal graph on which we can "cut $\Gamma$ in two" in order to obtain a new graph $\Gamma'$, used to parametrize $C$. Notice that on all the edges of $\Gamma_\text{even}$, the map
$h$ has an even slope.
\end{rem}

As $C$ admits at least one odd unbounded end, each connected component $(\Gamma_\text{even})_i$ of $\Gamma_\text{even}$ contains a unique \textit{stem}, which is a non-extendable vertex. We orient the edges of $(\Gamma_\text{even})_i$ away from the stem. Then we say that a subset of points $\mathcal{R}_i\subset(\Gamma_\text{even})_i$ is \textit{admissible} if no point of $\mathcal{R}_i$ is joint to another by an oriented path, and for each unbounded end $e$ in $(\Gamma_\text{even})_i$, there is at least (and thus exactly one) point of $\mathcal{R}_i$ on the shortest path between the stem and $e$. Let $\mathcal{R}=\bigcup_i\mathcal{R}_i$. We then define a real abstract tropical curve $(\Gamma(\mathcal{R}),\sigma)$ with a map $h_\mathcal{R}:\Gamma(\mathcal{R})\rightarrow N_\RR$ that factors through $\Gamma(\mathcal{R})\rightarrow \Gamma\rightarrow N_\RR$ and makes it a real parametrized tropical curve.\\

Let $\Gamma_\text{fix}(\mathcal{R})$ be the closure of the union of the connected components of $\Gamma-\mathcal{R}$ not containing any even end. The abstract tropical curve $\Gamma(\mathcal{R})$ is obtained as the disjoint union of two copies of $\Gamma$, glued along $\Gamma_\text{fix}(\mathcal{R})$:
$$\begin{tikzcd}
\Gamma_\text{fix}(\mathcal{R}) \arrow{r} \arrow{d} & \Gamma\arrow{d} \\
\Gamma \arrow{r} & \Gamma(\mathcal{R}) \\
\end{tikzcd}.$$

 In other terms, $\Gamma(\mathcal{R})=\Gamma\coprod_{\Gamma_\text{fix}(\mathcal{R})}\Gamma$. It means that we have doubled the components of $\Gamma-\mathcal{R}$ containing the even ends. We denote by $\pi:\Gamma(\mathcal{R})\rightarrow\Gamma$ the map obtained by gluing the identity maps of $\Gamma$. The complement of $\Gamma_\text{fix}(\mathcal{R})$ in $\Gamma$ is called the \textit{splitting graph}. It is a subset of $\Gamma_\text{even}$. The splitting graph is maximal if its closure is equal to $\Gamma_\text{even}$. The length function on $\Gamma(\mathcal{R})$ is defined as follows: we consider points of $\mathcal{R}$ as vertices of $\Gamma$, then, the length of an edge $\gamma$ of $\Gamma(\mathcal{R})$ is the length of its image $\pi(\gamma)$ if it is an edge of $\Gamma_\text{fix}(\mathcal{R})$ and twice the length of $\pi(\gamma)$ otherwise. The involution $\sigma$ is the automorphism of $\Gamma(\mathcal{R})$ that exchanges the two antecedents whenever there are two. The parametrized map $h_\mathcal{R}:\Gamma(\mathcal{R})\rightarrow N_\RR$ is the composition of $\pi$ and $h$.
 
\begin{rem}
The map $\pi$ really looks like a tropical cover, as defined in \cite{cavalieri2010tropical} and \cite{buchholz2015tropical}. However, it is not always the case. This is normal since the purpose of the notion of tropical cover is to mimick ramified covers between complex curves. The map $\pi$ here plays the role of the quotient map by a real involution, which is not a ramified cover.
\end{rem}

Let $\gamma$ be an edge of $\Gamma(\mathcal{R})$, and $n\in N$ be the slope of $\pi(\gamma)$. Then, one can easily check that the choice of length on $\Gamma(\mathcal{R})$ ensures that $h_\mathcal{R}$ has slope $n$ if $\gamma\in\fix$ and $\frac{n}{2}$ otherwise. However, as the edges of $\Gamma_\text{even}$ have an even slope, it is still an element of $N$. One can check that the balancing condition is still satisfied. Therefore, $(\Gamma(\mathcal{R}),h_\mathcal{R},\sigma)$ is a real parametrized tropical curve, of image $C$, and of degree $\{u_i,v_j^2\}_{i,j}$.

\begin{prop}\label{realstruc}
Let $C$ be an irreducible rational plane tropical curve of degree $P_\Delta\subset M$ having unbounded ends of weight $1$ or $2$. Let $\Delta\subset N$ be the degree associated to $P_\Delta$ consisting only of primitive lattice vectors. let $h:\Gamma\rightarrow N_\RR$ be the unique rational parametrization of $C$ given by Proposition \ref{unique parametrization}. Using previous notations, every real rational parametrized curve of degree $\Delta$ having the image $C$ is one of the curves $\Gamma(\mathcal{R})$.
\end{prop}

\begin{proof}
The curves $(\Gamma(\mathcal{R}),h)$ provide real rational parametrizations of $C$. Conversely, if $h:(\Gamma',\sigma)\rightarrow N_\RR$ is a real rational parametrization of $C$ of degree $\{u_i,v_j^2\}$, then we have quotient curve $\Gamma'/\sigma$ defined as follows. As a topological space, $\Gamma/\sigma$ is the quotient by $\sigma$. The edge lengths are the same for edges in $\fix$, and the edge length is divided by two for a pair of exchanged edges. Since $h$ is $\sigma$-invariant, we have a quotient map $\tilde{h}:\Gamma'/\sigma\rightarrow N_\RR$ and one can check that the above choice of edge length makes it into a parametrized tropical curve. The assumption on the weights of the unbounded ends of $C$ ensures that the conjugated ends of $\Gamma'$ are mapped to the even unbounded ends of $C$. Their weight is doubled when passing to the quotient. Thus, we get a rational parametrization of $C$ of degree $\{u_i,2v_j\}_{i,j}$. Therefore, it is isomorphic to $\Gamma$. Let $\pi:\Gamma\rightarrow\Gamma/\sigma$ be the quotient map.\\

The primitivity assumption on the degree ensures that near infinity, the points of the even unbounded ends of $\Gamma$ have two antecedents by $\pi$. The other ends only have one. Let $\mathcal{R}$ be the boundary of $\fix$. 
\begin{itemize}
\item First of all, $\fix$ is connected: if $p,q\in\fix$, there is a unique shortest path in $\Gamma$ between $p$ and $q$, this path is then $\sigma$-invariant, thus in $\fix$.
\item Let $\Xi$ be a connected component of $\Gamma\backslash\fix$, the boundary of $\Xi$ contains exactly one point of $\mathcal{R}$: at least one since $\Xi\neq\Gamma$, and at most one, otherwise the path between these points of $\mathcal{R}$ would lie in $\Xi$, and we have proven that such a path lies in $\fix$. Thus, $\Xi$ only contains even ends, and the construction of $\Gamma_\text{even}$ ensures that the point of $\mathcal{R}$ on the boundary of $\Xi$ is in $\Gamma_\text{even}$.
\item Finally, we have proven that $\Gamma$ is composed of $\fix$, which is connected and has boundary $\mathcal{R}$, and components $\Xi$ that are attached to $\fix$ at those vertices. Thus, the configuration $\mathcal{R}$ is admissible: there is at least one point of $\mathcal{R}$ on the shortest path between the stem and an even end since the stem is in $\fix$ and the end is not, and there is at most one since $\fix$ is connected.
\end{itemize}
Finally, the set $\mathcal{R}$ being admissible, the graph $\Gamma'$ is recovered as the curve $\Gamma(\mathcal{R})$.
\end{proof}

	\subsection{Moment of an edge}
	
	Recall that $\omega$ is a generator of $\bigwedge^2 M$, \textit{i.e.} a non-degenerated $2$-form on $N$. It extends to a volume form on $N_\RR\simeq\RR^2$. Let $e\in\Gamma^1_\infty$ be an unbounded end oriented toward infinity, directed by $n_e$. Then the moment of $e$ is the scalar
	$$\mu_e=\omega(n_e,p)\in\RR,$$
	where $p\in e$ is any point on the edge $e$. Remember that we identify the affine space $N_\RR$ with its tangent space at $0$, allowing us to plug in $\omega$ a tangent vector $n_e$ and a point $p$. We similarly define the moment of a bounded edge if we specify its orientation. The moment of a bounded edge is reversed when its orientation is reversed.\\
	
	Intuitively, the moment of an unbounded end is just a way of measuring its position alongside a transversal axis. Thus, fixing the moment of an unbounded end amounts to impose on the curve that it goes through some point at infinity. In a way, this allows us to do toric geometry in a compactification of $N_\RR$ but staying in $N_\RR$. It provides a coordinate on the components of the toric boundary without even having to introduce the concept of toric boundary in the tropical world. Following this observation, the moment has also a definition in complex toric geometry, where it corresponds to the coordinate of the intersection point of the curve with the toric divisor. Let

$$\begin{array}{rccl}
\varphi: & \CC P^1 & \dashrightarrow & \text{Hom}(M,\CC^*)\simeq(\CC^*)^2 \\
 & t & \mapsto & \chi\prod_{1}^r(t-\alpha_j)^{n_j}. \\
 \end{array}$$
	 be a parametrized rational curve. Given a dual basis $(e_1,e_2)$ of $N$, and its dual basis $(e_1^*,e_2^*)$ of $M$, the parametrized curve given in coordinates is as follows. Let $n_i=a_ie_1+b_ie_2$, $a=\chi(e_1^*)$ and $b=\chi(e_2^*)$, then
	 $$\varphi(t)=\left( a\prod_1^r(t-\alpha_i)^{a_i},b\prod_1^r(t-\alpha_i)^{b_i}\right)\in (\CC^*)^2.$$
	 This is a curve of degree $\Delta=(n_j)\subset N$. The degree $\Delta$ defines a fan $\Sigma_\Delta$ and a toric surface $\CC\Delta$ to which the map $\varphi$ naturally extends. The toric divisors $D_k$ of $\CC\Delta$ are in bijection with the rays of the fan, which are directed by the vectors $n_j$. Several vectors $n_j$ may direct the same ray. Moreover, the map $\varphi$ extends to the points $\alpha_j$ by sending $\alpha_j$ to a point on the toric divisor $D_k$ corresponding to the ray directed by $n_j$. A coordinate on $D$ is a primitive monomial $\chi^m\in M$ in the lattice of characters such that $\langle m,n_j\rangle=0$. This latter equality ensures that the monomial $\chi^m$ extends on the divisor $D_k$. If $n_j$ is primitive, $\iota_{n_j}\omega\in M$ is such a monomial, and then the complex moment is the evaluation of the monomial at the corresponding point on the divisor:
	$$\mu_j=\left(\varphi^*\chi^{\iota_{n_j}\omega}\right)(\alpha_j).$$
	
	The Weil reciprocity law gives us the following relation between the moments:
$$\prod_{i=1}^m \mu_i=(-1)^m.$$	
	We could also prove the relation using Vi\`ete formula. In the tropical world we have an analog called the \textit{tropical Menelaus theorem}, which gives a relation between the moments of the unbounded ends of a parametrized tropical curve.
	
	\begin{prop}[Tropical Menelaus Theorem \cite{mikhalkin2017quantum}] For a parametrized tropical curve of degree $\Delta$, we have
	$$\sum_{n_e\in\Delta} \mu_e =0.$$
	\end{prop}
	
	In the tropical case as well as in the complex case, a configuration of $m$ points on the toric divisors is said to satisfy the \textit{Menelaus condition} if this relation is satisfied. Be careful that in tropical case, the moment of a point depends on the vector of $M$ used to compute it: it is of lattice length $1$ if the point is real, and of lattice length $2$ for a non-real point.
	
	\subsection{Moduli space of tropical curves and refined multiplicity of a simple tropical curve}
	
	Let $(\Gamma,h)$ be a parametrized tropical curve such that $\Gamma$ is trivalent, and has no \textit{flat vertex}. A flat vertex is a vertex whose outgoing edges have their slope contained in a common line. It just means that for any two outgoing edges of respective slopes $u,v$, we have $\omega(u,v)=0$. In particular, when the curve is trivalent, no edge can have a zero slope since it would imply that its extremities are flat vertices. A plane tropical curve is a \textit{simple nodal curve} if the dual subdivision of its Newton polygon consists only of triangles and parallelograms. The unique rational parametrization (given by Proposition \ref{unique parametrization}) of a plane rational nodal curve has a trivalent underlying graph, and has no flat vertex.\\
	
	\begin{defi}
	The \textit{combinatorial type} of a tropical curve is the homeomorphism type of its underlying labeled graph $\Gamma$, \textit{i.e.} the labeled graph $\Gamma$ without the metric.
	\end{defi}
	
	To give a graph a tropical structure, one just needs to specify the lengths of the bounded edges. If the curve is trivalent and has $m$ unbounded ends, there are $m-3$ bounded edges, otherwise the number of bounded edges is $m-3-\text{ov}(\Gamma)$, where $\text{ov}(\Gamma)$ is the \textit{overvalence} of the graph. The overvalence is given by $\sum_V (\text{val}(V)-3)$, where $V$ runs over the vertices of $\Gamma$, and $\text{val}(V)$ denotes the valence of the vertex. Therefore, the set of curves having the same combinatorial type is homeomorphic to $\RR_{\geqslant 0}^{m-3-\text{ov}(\Gamma)}$, and the coordinates are the lengths of the bounded edges. If $\Gamma$ is an abstract tropical curve, we denote by $\text{Comb}(\Gamma)$ the set of curves having the same combinatorial type as $\Gamma$.\\
	
	For a given combinatorial type $\text{Comb}(\Gamma)$, the boundary of $\RR_{\geqslant 0}^{m-3-\text{ov}(\Gamma)}$ corresponds to curves for which the length of an edge is zero, and therefore corresponds to a graph having a different combinatorial type. This graph is obtained by deleting the edge with zero length and merging its extremities. We can thus glue together all the cones of the finitely many combinatorial types and obtain the \textit{moduli space $\mathcal{M}_{0,m}$ of rational tropical curves with $m$ marked points}. It is a simplicial fan of pure dimension $m-3$, and the top-dimensional cones correspond to trivalent curves. The combinatorial types of codimension $1$ are called \textit{walls}.\\
	
	Given an abstract tropical curve $\Gamma$, if we specify the slope of every unbounded end, and the position of a vertex, we can define uniquely a parametrized tropical curve $h:\Gamma\rightarrow N_\RR$. Therefore, if $\Delta\subset N$ denotes the set of slopes of the unbounded ends, the \textit{moduli space $\mathcal{M}_0(\Delta,N_\RR)$ of parametrized rational tropical curves of degree 
$\Delta$} is isomorphic to $\mathcal{M}_{0,m}\times N_\RR$ as a fan, where the $N_\RR$ factor corresponds to the position of the finite vertex adjacent to the first unbounded end.\\
	
	On this moduli space, we have a well-defined evaluation map that associates to each parametrized curve the family of moments of its unbounded ends :
	$$\begin{array}{crcl}
	\text{ev} : & \mathcal{M}_0(\Delta,N_\RR) & \longrightarrow & \RR^{m-1} \\
	 & (\Gamma,h) & \longmapsto & \mu=(\mu_i)_{2\leqslant i\leqslant m}
	\end{array}.$$
	By the tropical Menelaus theorem, the moment $\mu_1$ is equal to the opposite of the sum of the other moments, hence we do not take it into account in the map. Notice that the evaluation map is linear on every cone of $\mathcal{M}_0(\Delta,N_\RR)$. Furthermore, 
both spaces have the same dimension $m-1$. Thus, if $\Gamma$ is a trivalent curve, the restriction of $\text{ev}$ on
$\text{Comb}(\Gamma)\times N_\RR$ has a determinant well-defined up to sign when $\RR^{m-1}$ and $\text{Comb}(\Gamma)\simeq\RR_{\leqslant0}^{m-3}$ are both endowed with their canonical basis, and $N_\RR$ is endowed with a basis of $N$. 
The absolute value $m_\Gamma^\CC$ of the determinant is called the \textit{complex multiplicity} of the curve, well-known to factor into the following product over the vertices of $\Gamma$:
	$$m_\Gamma^\CC=\prod_V m_V^\CC,$$
	where $m_V^\CC=|\omega(u,v)|$ if $u$ and $v$ are the slopes of two outgoing edges of $V$. The balancing condition ensures that $m_V^\CC$ does not depend on the chosen edges. This multiplicity is the one that appears in the correspondence theorem of Mikhalkin \cite{mikhalkin2005enumerative}. Notice that the simple parametrized tropical curves are precisely the points of the cones with trivalent graph and non-zero multiplicity. We finally recall the definition of the refined Block-G\"ottsche multiplicity.
	
	\begin{defi}
	The \textit{refined multiplicity} of a simple nodal tropical curve is
	$$m^q_\Gamma=\prod_V [m_V^\CC]_q,$$
	where $[a]_q=\frac{q^{a/2}-q^{-a/2}}{q^{1/2}-q^{-1/2}}$ is the $q$-analog of $a$.
	\end{defi}
	
	This refined multiplicity is sometimes called the Block-G\"ottsche multiplicity and intervenes in the definition of the invariant $N_\Delta^{\partial,\text{trop}}$. Notice that the multiplicity is the same for every curve inside a given combinatorial type.

	\subsection{Tropicalization}
	\label{tropicalization}

We briefly recall how to obtain an abstract tropical curve and a parametrized tropical curve from a non-archimedean parametrized curve given by a rational map $f:(C,\textbf{q})\rightarrow\text{Hom}(M,\CC((t))^*)$, where $(C,\textbf{q})$ is a curve with marked points. For more details, see \cite{tyomkin2012tropical}.\\

\subsubsection{Tropicalization of a marked curve} 
 
Let $(C,\textbf{q})$ be a smooth marked curve over $\CC((t))$. Let $\mathcal{C}^{(t)}\rightarrow\text{Spec}\CC[[t]]$ be the stable model of $(C,\textbf{q})$, defined over $\CC[[t]]$. The marked points $q_i$ provide sections $\text{Spec}\CC[[t]]\rightarrow\mathcal{C}^{(t)}$. We have the special fiber $\mathcal{C}^{(0)}$, which is a stable nodal curve, meaning that each irreducible component of genus zero has at least three marked points or nodes, and each irreducible component of genus 1 has at least one marked point or a node. Let $\overline{\Gamma}$ be the dual graph in the following sense: we have one finite vertex per irreducible component of the special fiber, one infinite vertex per marked point, an infinite vertex is joined to the finite vertex of the component where the point specializes, and two finite vertices are joined by an edge if they share a node. We make $\overline{\Gamma}$ into an abstract tropical curve by declaring the length of such an edge to be $l$ if the node is locally given by $xy=t^l$ in an etale neighborhood of the node.

\begin{rem}
Intuitively, if our curve is rational, $\mathcal{C}^{(t)}$ is just $\CC P^1$ with points depending on a small complex parameter $t$ on it, \textit{i.e.} points given by locally convergent Laurent series in $\CC((t))$. If we take naively the special fiber $t=0$, some marked points may collide, \textit{i.e.} specialize on the same point, other may go to infinity, ... Taking the stable model means that we prevent that. For instance, assume a bunch of points specialize to $0$. This means that they are of given by formal series of the form $t^k x(t)$ with $k>0$. We then blow-up this point and get two copies of $\CC P^1$ sharing a node. All the points previously specializing to $0$ now specialize to at least two different points on the exceptional divisor. The length of the edge between the two copies is the smallest $k$ for all the points specializing on it. Concretely, the blow-up amounts to change the coordinate $z$ on $\CC P^1$ by $t^{-k}z$. We then repeat as long as necessary. 
\end{rem}

If the curve $(C,\textbf{q})$ is a real curve, with a real configuration of points $\textbf{q}$, the involution restricted to the special fiber induces a real structure on $\Gamma$.


\subsubsection{Tropicalization of a parametrized curve} Now assume given a rational map $f:(C,\textbf{q})\dashrightarrow\text{Hom}(M,\CC((t))^*)$. There is a tropical curve $\Gamma$ associated to $(C,\textbf{q})$. The rational map $f$ extends to a rational map on the stable model $\mathcal{C}^{(t)}$ of $(C,\textbf{q})$. In order to make $\Gamma$ into a parametrized tropical curve, we define a map $h:\Gamma\rightarrow N_\RR$ in the following way:
\begin{itemize}
\item If $w\in\Gamma^0$ is a vertex dual to a component $C_w$ of $\mathcal{C}^{(0)}$, then $h(w)$ is the element of $N$ defined as follows:
$$h(w)(m)=\text{ord}_{C_w}(f^*\chi^m),$$
where $\text{ord}_{C_w}$ stands for the multiplicity of $C_w$ in the divisor of $f^*\chi^m$.
\item Then $h$ maps a bounded edge to the line segment linking its extremities.
\item If $q_i$ is a marked point, then the slope of the associated unbounded end is $\text{ord}_{q_i}(f^*\chi^m)$, where $\text{ord}_{q_i}$ stands for the multiplicity of $q_i$ in the divisor of $f^*\chi^m$.
\end{itemize}

\begin{rem}
The slope of the unbounded end associated to a given marked point $q_i$ is given both by the order of vanishing of $f^*\chi^m$ at $q_i$, and by the multiplicity of the section defined by $q_i$ in the divisor of $f^*\chi^m$ in the stable model $\mathcal{C}^{(t)}$. This is normal since the marked points provide divisors in $\mathcal{C}^{(t)}$ which are transverse to the special fiber. Concretely, in the rational case, if $y$ is a coordinate on $C$ and $f$ is given by
$$f:y\longmapsto\chi\prod_{i=1}^r(y-y(q_i))^{n_i}\in\text{Hom}(M,\CC((t))^*),$$
with $\chi\in\text{Hom}(M,\CC((t))^*)$, then the slope of the edge associated to the marked point $q_i$ is $n_i$.
\end{rem}

\begin{rem}
The special fiber is given by the equation $t=0$. Therefore $h(w)(m)=\text{ord}_{C_w}(f^*\chi^m)$ is the valuation in $t$ of the function evaluated at the generic point of $C_w$. Concretely, in the rational case, let $y$ be a coordinate on $C$ specializing to a coordinate on $C_w$, which is a copy of $\CC P^1$, such that no point specializes to $\infty$. Assume $f$ is given by
$$f:y\longmapsto\chi\prod_{i=1}^r(y-y(q_i))^{n_i}\in\text{Hom}(M,\CC((t))^*),$$
where $\chi\in\text{Hom}(M,\CC((t))^*)$. Then $h(w)(m)=\val(\chi(m))$.
\end{rem}

The fact that $(\Gamma,h)$ is indeed a parametrized tropical curve is proved in \cite{tyomkin2012tropical}. One essentially needs to check the balancing condition, and the fact that if $\gamma$ is an edge with extremities $v$ and $w$, the slope of $h(\gamma)$ lies in $N$, and the length of $h(w)-h(v)$ coincide with the length $|\gamma|$ in $\Gamma$. Finally, we can refine the tropicalization in the following way that is useful to compute quantum indices: on each $C_w$ the rational map $f$ specializes to give a parametrized complex rational curve $f_w:C_w\simeq\CC P^1\dashrightarrow\text{Hom}(M,\CC^*)$. Concretely, this is the curve we would obtain by taking the naive limit of $f$ in a coordinate specializing to a coordinate of $C_w$. Therefore, we have a tropical curve and a complex curve associated to every vertex.

\begin{rem}
All our curves are taken with coefficients in $\CC((t))$, which is not algebraically closed, and has a discrete valuation. Thus, every tropicalization data has coefficients in $\ZZ$. Instead we could take the algebraic closure, which is the field of Puiseux series $\CC\{\{t\}\}=\bigcup_{k\geqslant 1}\CC((t^\frac{1}{k}))$, but as we are using only a finite number of coefficients, all belong to $\CC((t^\frac{1}{k}))$ for some $k$, and by taking $u=t^\frac{1}{k}$ we reduce it to the previous case. Therefore, we can assume that everything is defined in $\CC((t))$, up to a change of base.
\end{rem}

\subsubsection{Tropicalization of a plane curve}

We finish by describing the tropicalization of a plane curve. This tropicalization is more elementary than the tropicalization of a parametrized curve. Moreover, the tropicalization of a parametrized curve gives a parametrization of the tropicalization of its image plane curve. Let $C$ be a plane curve, defined by a polynomial $P_t\in\CC((t))[M]$ with coefficients in $\CC((t))$. We look for the points of the curve over the Puiseux series, \textit{i.e.} in $N\otimes\CC\{\{t\}\}^*$. In a basis of $M$, the polynomial is given in coordinates by
$$P_t(x,y)=\sum_{(i,j)\in P_\Delta}a_{i,j}(t)x^iy^j.$$
We assume that the coefficients in the corners of $P_\Delta$ are non-zero. Then, we have the associated tropical polynomial
$$\text{Trop}(P_t)(x,y)=\max_{(i,j)\in P_\Delta}\left(\val(a_{i,j}(t))+ix+jy\right),$$
along with a valuation map, also called \textit{tropicalization map}:
$$\text{Val}:\chi\in \text{Hom}(M,\CC\{\{t\}\}^*)\longmapsto\text{val}\circ \chi\in \text{Hom}(M,\RR)=N_\RR.$$
In coordinates, $\text{Val}$ is given by the coordinatewise valuation:
$$\text{Val}:(x,y)\in(\CC\{\{t\}\}^*)^2\longmapsto (\val(x),\val(y))\in\RR^2.$$
The Kapranov theorem \cite{brugalle2014bit} then ensures that the closure of the image of the vanishing locus of $P_t$ in $(\CC\{\{t\}\}^*)^2$ under the valuation map is equal to the tropical curve defined by $\text{Trop}(P_t)$.

\begin{theo}[Kapranov]
Let $C_\text{trop}$ be the tropical curve defined by $\text{Trop}(P_t)$. Then, one has
$$\overline{\text{Val}(C)}=C_\text{trop}.$$
\end{theo}

Let $\alpha_{i,j}=\val(a_{i,j}(t))$, and $a_{i,j}(t)=t^{\alpha_{i,j}}a_{i,j}^0(t)$. The function $(i,j)\mapsto\alpha_{i,j}$ induces a convex subdivision of $P_\Delta$, which is dual to $C_\text{trop}$. As in the tropicalization of a parametrized curve, one can recover complex curves, by specializing the polynomial $P_t$ to one of the polygons of the subdivision. Let $\varpi$ be one of the polygons of the subdivision of $P_\Delta$. Then, the curve associated to $\varpi$ is given by $P_\varpi(x,y)=\sum_{(i,j)\in\varpi}a_{i,j}^0(0)x^iy^j=0$, defined over $\CC$.\\

One can show that if $\varphi:C\rightarrow N\otimes\CC\{\{t\}\}^*$ is a parametrized curve tropicalizing to $h:\Gamma\rightarrow N_\RR$, then the image $h(\Gamma)$ and the tropicalization of the image $\overline{\text{Val}(\varphi(C))}$ are the same. Moreover, the local parametrized curves $f_w:C_w\dashrightarrow N\otimes \CC^*$ resulting from the tropicalization as parametrized curve, are precisely the irreducible components of the curves defined by $P_\varpi=0$. 

\section{Quantum indices of real tropical curves}

We start this section by recalling the theorem about quantum indices by Mikhalkin \cite{mikhalkin2017quantum}, restricting ourselves to the case of rational curves. We then compute the quantum indices of some specific curves that appear in the resolution of our enumerative problem.

	\subsection{The quantum index of a type $I$ real curve}

Let
$$\varphi:t\in\CC\longmapsto \chi\prod_1^r(t-\alpha_i)^{n_i}\prod_1^s(t-\beta_j)^{n_j}(t-\overline{\beta_j})^{n_j}$$
be a parametrized real rational curve, with $\alpha_i\in\RR$, $\beta_j\in\CC\backslash\RR$, and $\chi\in N\otimes\CC^*$. Recall that the moment of the parametrized curve $(\CC P^1,\varphi)$ at a complex point $\beta_{j_0}$ is the quantity
$$\varphi^*\chi^{\iota_{n_{j_0}}\omega}|_{\beta_{j_0}}= \chi(\iota_{n_{j_0}}\omega)\prod_1^r(\beta_{j_0}-\alpha_i)^{\omega(n_{j_0},n_i)}\prod_1^s(\beta_{j_0}-\beta_j)^{\omega(n_{j_0},n_j)}(\beta_{j_0}-\overline{\beta_j})^{\omega(n_{j_0},n_j)}\in\CC^*.$$

\begin{defi}
In the above notations, we say that the rational curve has real or purely imaginary intersection points if $\varphi^*\chi^{\iota_{n_{j_0}}\omega}|_{\beta_{j_0}}\in i\RR$ for every $j_0$.
\end{defi}

\begin{rem}
At the real points $\alpha_i$, the moment is real since the function $\varphi$ is real, that is why we only check non-real points for the purely imaginary value. Geometrically, 
it means that the coordinates of the intersection points of the curve with the toric boundary are either real or purely imaginary. In both cases their square is real.
\end{rem}

Recall that we have the logarithmic map
$$\text{Log}:n\otimes z\in N\otimes\CC^*\longmapsto n\otimes\text{Log}|z|\in N\otimes\RR.$$
In a basis of $N$, it is the logarithm of the absolute value coordinate by coordinate. Similarly we define the argument map
$$2\arg :n\otimes z\in N\otimes\CC^*\longmapsto n\otimes 2\arg(z)\in N\otimes\RR/\pi\ZZ.$$
The parametrized real rational curve $\varphi:\CC P^1\rightarrow N\otimes \CC^*$ is of type $I$. Let $S$ be a connected component of $\CC P^1\backslash\RR P^1$, inducing a complex orientation of $\RR P^1$. By pulling back the volume form $\omega$ on $N_\RR$ to $N\otimes\CC^*$, we can define the logarithmic area of $S$:
$$\mathcal{A}_{\text{Log}}(S)=\int_{\varphi(S)}\text{Log}^*\omega.$$
Respectively, the $2$-form $\omega$ defines a $2$-form $\omega_\theta$ on $N\otimes\RR/\pi\ZZ$. We can pull it back to $N\otimes\CC^*$ and define the area of the co-amoeba of $S$:
$$\mathcal{A}_{\arg}(S)=\int_{\varphi(S)}(2\arg)^*\omega.$$
If $\omega$ is given in coordinates by $\omega=\text{d}x_1\wedge\text{d}x_2$, due to the vanishing of the meromorphic $2$-form $\frac{\text{d}z_1}{z_1}\wedge\frac{\text{d}z_2}{z_2}$ restricted to $S$, one has $\mathcal{A}_{\text{Log}}(S)=\mathcal{A}_{2\arg}(S)$.

\begin{theo}[Mikhalkin\cite{mikhalkin2017quantum}]
Let $\varphi:\CC P^1\dashrightarrow N\otimes\CC^*$ be a real parametrized rational curve with real or purely imaginary intersection points, enhanced with the choice of a connected component $S$ of $\CC P^1\backslash\RR P^1$, inducing a complex orientation of $\RR P^1$. Then there exists a half-integer $k(S,\varphi)$ such that
$$\mathcal{A}_{\arg}(S) = \mathcal{A}_{\text{Log}}(S) = k(S,\varphi)\pi^2.$$
\end{theo}

	\subsection{The quantum index near the tropical limit}
	
	In \cite{mikhalkin2017quantum}, Mikhalkin proved the following result, that computes the quantum index of curves in a family near the tropical limit.
	
	\begin{prop}\cite{mikhalkin2017quantum}
	Let $C^{(t)}=\left(f_t:\CC P^1\rightarrow\text{Hom}(M,\CC^*)\right)$ be a family of type I real parametrized rational curves, having real or purely imaginary intersection points, enhanced with a family of connected components of the complex locus $S^{(t)}$, inducing complex orientations of the curves. We assume that the family tropicalizes, in the sense of \ref{tropicalization}, to a parametrized tropical curve $h:\Gamma\rightarrow N_\RR$, such that components $S^{(t)}$ specialize to components $S_w$ of $C_w$ for every vertex $w\in\fix$, thus inducing complex orientations of the curves $C_w$. Then, all the curves $f_w:C_w\rightarrow\text{Hom}(M,\CC^*)$ have real or purely imaginary intersection points, and moreover, for $t$ large enough,
	$$k(S^{(t)},f_t)=\sum_w k(S_w,f_w),$$ 
where the sum is indexed over the fixed vertices of $\Gamma$.
\end{prop}

\begin{rem}
In particular, and this happens in the proof of the correspondence theorem, for one to know the quantum index of curves near the tropical limit, one only needs to know the quantum indices of the curves associated to the vertices of the tropical curve, and the way they are glued together along the edges. This means that the quantum index may be computed in the patchworking construction. 
\end{rem}

The computation near the tropical limit allows us to reduce the calculations needed to compute the quantum indices of oriented curves: we only need to compute the quantum indices of oriented curves associated to the vertices of a tropical curve. This includes real rational curves with three real intersection points, and real rational curves with two real and two complex conjugated intersection points. The following lemma reduces the computation of these two cases to two computations, dealt with in the next subsection.\\

We prove that the quantum index is well-behaved under the monomial maps, which are covering maps from the complex torus to itself. This allows us to restrict the already reduced computations to a few cases easier to compute.

\begin{lem}
Let $\varphi : \CC C\dashrightarrow\text{Hom}(M,\CC^*)$ be a type $I$ real curve with a choice of a connected component $S\subset \CC C\backslash\RR C$, inducing a complex orientation, and let $\alpha:\text{Hom}(M,\CC^*)\rightarrow \text{Hom}(M',\CC^*)$ be a monomial map, associated to a morphism $A^T : M'\rightarrow M$. We consider the composition
$$\psi : \CC C \stackrel{\varphi}{\longrightarrow}\text{Hom}(M,\CC^*) \stackrel{\alpha}{\longrightarrow}\text{Hom}(M',\CC^*).$$
Let $\omega$ and $\omega'$ be the volume forms on respectively $N$ and $N'$, dual lattices of $M$ and $M'$, so that we have $A^*\omega'=\det A\omega$. Then, we have
$$\int_{\psi(S)}\text{Log}^*\omega=\det A \int_{\varphi(S)}\text{Log}^*\omega \text{ and } \int_{\psi(S)}(2\arg)^*\omega_\theta = \det A \int_{\varphi(S)}(2\arg)^*\omega_\theta.$$
\end{lem}

\begin{rem}
The proposition deals with the computation of log-area in the general case of a real curve. This log-area is a quantum index only if the curve has real or purely imaginary intersection points with the toric boundary. In many cases, we use this proposition to reduce the computation of a quantum index to a log-area of a curve which does not necessarily have a quantum index.
\end{rem}

\begin{rem}
Notice that the different notations $\text{Hom}(M,\CC^*)$ and $\text{Hom}(M',\CC^*)$ prevent any mistakes in the direction of the various involved maps.
\end{rem}

\begin{proof}
Let $N$ and $N'$ be the dual lattices of $M$ and $M'$, so that we have linear maps
$$A^T:M'\rightarrow M,
A:N\rightarrow N'.$$
Then, we have the following commutative diagrams:
$$\begin{array}{ccc}
\begin{tikzcd}
    \CC C \arrow{dr}{\psi} \arrow{d}{\varphi} &  \\
     N\otimes\CC^* \arrow{d}{\text{Log}} \arrow{r}{\alpha}& N'\otimes \CC^* \arrow{d}{\text{Log}} \\
    N_\RR \arrow{r}{A} &  N'_\RR \\
  \end{tikzcd}
  & \text{ and } &
  \begin{tikzcd}
    \CC C \arrow{dr}{\psi} \arrow{d}{\varphi} &  \\
     N\otimes\CC^* \arrow{d}{\arg} \arrow{r}{\alpha}& N'\otimes \CC^* \arrow{d}{\arg} \\
    N\otimes\RR/2\pi\ZZ \arrow{r}{A} &  N'\otimes\RR/2\pi\ZZ \\
  \end{tikzcd} \\
  \end{array}.$$
We denote by $\omega$, $\omega'$ the volume forms of the lattices $N$ and $N'$ used to compute the log-areas. Then, we have
\begin{align*}
\int_{\psi(S)}\text{Log}^*\omega' & =\int_S (\text{Log}\circ\psi)^*\omega' \\
 & = \int_S (A\circ\text{Log}\circ\varphi)^*\omega' \\
 & = \int_S (\text{Log}\circ\varphi)^*(A^*\omega') \\
 & = \det A\int_S (\text{Log}\circ\varphi)^*\omega \text{ since }A^*\omega'=(\det A)\omega,\\
 & = \det A\int_{\varphi(S)}\text{Log}^*\omega.\\
\end{align*}
Proof is completely similar for the argument maps.
\end{proof}

	\subsection{Local computations}
\label{local computations}

In this section, we compute the quantum indices of some auxiliary rational curves. This includes complex conjugated rational curves, a real rational curve with three intersection points with the toric boundary, and a real rational curve with two real and two complex intersection points with the toric boundary.

\subsubsection{Log-area of a complex curve}

We begin by proving that the log-area of a complex curve is zero. This proves that the non-fixed vertices of the tropical curve have no contribution to the quantum index. This justifies the fact that the quantum index near the tropical limit is obtained as a sum over the fixed vertices, and not the pairs of exchanged vertices. The following statement is not specific to rational curves or real curves. 

\begin{lem}
Let $\varphi:\CC C\dashrightarrow N\otimes \CC^*$ be a complex parametrized curve, with $\CC C$ a smooth Riemann surface. Then
$$\int_{\CC C}\text{Log}^*\omega=\int_{\CC C}(2\arg)^*\omega_\theta=0.$$
\end{lem}

\begin{proof}
The two integrals are known to be equal by the vanishing of the holomorphic $2$-form given in coordinates by $\frac{\dd z_1}{z_1}\wedge\frac{\dd z_2}{z_2}$. Let $\CC C^o$ be the open set of $\CC C$ where $\varphi$ is defined. We consider the map $\text{Log}\circ\varphi:\CC C^o\rightarrow N_\RR$. This is a proper map between smooth oriented manifolds, therefore it has a well-defined degree, which corresponds both to the number of antecedents counted with signs over a generic point, and to the map $\RR =H^2_c(N_\RR)\stackrel{ (\text{Log}\circ\varphi)^*}{\longrightarrow} H^2_c(\CC C^o)=\RR$ between compactly supported cohomology groups. Since the map is not surjective, its degree is zero. Hence, if $\tilde{\omega}$ is a compactly supported $2$-form on $N_\RR$, then $\int_{\CC C^o}(\text{Log}\circ\varphi)^*\tilde{\omega}=0$. Thus, by writing $\omega$ as a (infinite) sum of compactly supported $2$-forms using partitions of unity, we get the result.
\end{proof}

\subsubsection{Trivalent real vertex} We first recall the computation of the quantum index of a rational curve with three real punctures. This was dealt with in \cite{mikhalkin2017quantum}.

\begin{lem}
Let $\Delta\subset N$ be a family of three vectors of total sum $0$, and let $P_\Delta\subset M$ be the associated triangle, of lattice area $m_\Delta$. Let $\varphi : \CC P^1\dashrightarrow N\otimes\CC^*$ be a real parametrized rational curve of degree $\Delta$, thus having a unique real intersection point with maximal tangency with each toric divisor. Then, the quantum index of the curve is $\pm \frac{m_\Delta}{2}$ according to the choice of complex orientation.
\end{lem}

\begin{proof}
The assumption implies that the curve is the image of a line by a monomial map of determinant $m_\Delta$. Hence, its quantum index is the quantum index of a line, equal to $\pm\frac{1}{2}$, times the determinant of the monomial map which is the lattice area of the triangle.
\end{proof}

\subsubsection{Quadrivalent complex vertex} We now consider the case of a curve associated with a quadrivalent vertex having two edges exchanged by the involution $\sigma$, and two edges fixed. This means that this is a rational curve having two real punctures, and two conjugated ones. In a suitable choice of coordinates, the curve has a degree of the following form. In a basis $(e_1,e_2)$ of $N$, for $m_1,m_2,m_3\in\NN^*$, let us take
$$\Delta(m_1,m_2,m_3)=\{(m_1,2m_2);(0,m_3-m_2)^2;(-m_1,-2m_3)\}.$$
The degree of a planar curve which is parametrized by a curve of degree $\Delta(m_1,m_2,m_3)$ is the lattice polygon in $M$ given by
$$P_\Delta(m_1,m_2,m_3)=\text{Conv}\left( (0,m_1),(2m_2,0),(2m_3,0) \right).$$

\begin{figure}
\begin{center}
\definecolor{cqcqcq}{rgb}{0.7529411764705882,0.7529411764705882,0.7529411764705882}
\begin{tikzpicture}[line cap=round,line join=round,>=triangle 45,x=0.4cm,y=0.4cm]
\draw [color=cqcqcq,, xstep=0.4cm,ystep=0.4cm] (-1.,-1.) grid (7.,5.);
\clip(-1.,-1.) rectangle (7.,5.);
\draw (0.,4.)-- (2.,0.);
\draw (2.,0.)-- (6.,0.);
\draw (6.,0.)-- (0.,4.);
\begin{scriptsize}
\draw [fill=black] (0.,4.) circle (0.5pt);
\draw[color=black] (0.14,4.5) node {$(0,m_1)$};
\draw [fill=black] (2.,0.) circle (0.5pt);
\draw[color=black] (1.14,-0.5) node {$(2m_2,0)$};
\draw[color=black] (0.0,1.7) node {$E_3$};
\draw [fill=black] (6.,0.) circle (0.5pt);
\draw[color=black] (5.7,-0.5) node {$(2m_3,0)$};
\draw[color=black] (3.9,-0.66) node {$E_1$};
\draw[color=black] (3.5,2.44) node {$E_2$};
\end{scriptsize}
\end{tikzpicture}

\caption{\label{PDelta}The polygon $P_\Delta(m_1,m_2,m_3)$.}

\end{center}
\end{figure}
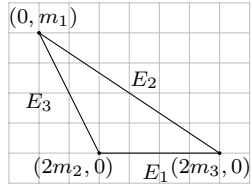

This polygon is drawn on Figure \ref{PDelta}. Up to an automorphism of the lattice, every triangle in $M$ having a side of even length is one of the polygons $P_\Delta(m_1,m_2,m_3)$. Let $E_i$ be the side opposite to the $i$-th vertex in $P_\Delta(m_1,m_2,m_3)$, \textit{i.e.}
$$E_1=\left[ (2m_2,0),(2m_3,0) \right] ,\ E_2=\left[ (2m_3,0),(0,m_1) \right] , \ E_3=\left[ (2m_2,0),(0,m_1) \right].$$
We denote by $\CC E_i$ the associated toric divisor inside the toric surface $\CC\Delta(m_1,m_2,m_3)$.\\

We take a real point on $\CC E_3$ and two purely imaginary conjugated points on $\CC E_1$, and look for real rational curves of degree $\Delta(m_1,m_2,m_3)$, maximally tangent to each toric divisor at the given points. The Menelaus theorem ensures that there exists a unique point on $\CC E_2$ such that each curve passing through the three chosen points also pass through the point on $\CC E_2$. Such a curve has a parametrization of the form
$$\varphi(t)=\left( a(t-c)^{m_1},b(t-c)^{2m_2}(t^2+1)^{m_3-m_2} \right)\in(\CC^*)^2,$$
where $c$ is some real number corresponding to the coordinate of the intersection point with $\CC E_3$, and $a,b\in\RR^*$. The intersection point with $\CC E_2$ corresponds to the coordinate $t$ taking the infinite value. The condition to pass through the specific points are given by the following equations:
$$a(i-c)^{m_1}=i\lambda\in i\RR^* \text{ and }\frac{b^\frac{m_1}{\delta}}{a^{\frac{2m_2}{d}}}(c^2+1)^{\frac{m_1}{\delta}(m_3-m_2)}=\mu\in\RR^*,$$
where $\delta=m_1\wedge(2m_2)$ is the integer length of $E_3$. The first equation solves for $c$ and $a$, and the second equation solves for $b$ with a unique solution if $\frac{m_1}{\delta}$ is odd, and $0$ or $2$ solution according to the sign of $\mu$ when it is even. The first equation implies that $(i-c)^{m_1}\in i\RR$ and thus we can write it $i-c=re^{i\pi\frac{2k+1}{2m_1}}$ with $k\in\ZZ$ and $r\in\RR$. Therefore, we have $i-re^{i\pi\frac{2k+1}{2m_1}}=c\in\RR$. Hence,
$$\begin{array}{rcl}
\im\left( i-re^{i\pi\frac{2k+1}{2m_1}} \right)=1-r\sin\left( \pi\frac{2k+1}{2m_1} \right)= 0 & \Rightarrow & r=\frac{1}{\sin\left( \pi\frac{2k+1}{2m_1} \right)} \\
 & \Rightarrow & c=r\cos\left( \pi\frac{2k+1}{2m_1} \right) = \cot\left( \pi\frac{2k+1}{2m_1} \right)=c_k.\\
\end{array}$$
We have proven that $c$ can only take a finite number of values $c_k=\cot\left( \pi\frac{2k+1}{2m_1} \right)$, for $k\in[\![0;m_1-1]\!]$. For each value of $c_k$ we find a unique $a$, and then solve for $b$ eventually. Thus, we have proven that up to the action of the real torus $(\RR^*)^2$, every real curve having purely imaginary intersection with $\CC E_1$, and real intersection with both $\CC E_2$ and $\CC E_3$ is one of the curves
$$\psi_k:t\longmapsto \left( (t-c_k)^{m_1},(t-c_k)^{2m_2}(t^2+1)^{m_3-m_2}\right).$$
These parametrized curves are the respective images of the curves
$$\varphi_k:t\longmapsto\left( t-c_k,\frac{t^2+1}{t-c_k}\right),$$
by the monomial map $\alpha:(z,w)\mapsto(z^{m_1},z^{m_3+m_2}w^{m_3-m_2})$. Therefore, in order to compute the quantum indices of the oriented curves $\psi_k$, we just need to compute the Log-areas of the oriented curves $\varphi_k$.

\begin{lem}
Let $\mathbb{H}$ denotes the half-plane $\{\im\ t>0\}$, inducing a complex orientation of $\RR P^1$. For any $k\in[\![0;m_1-1]\!]$, we have
$$\int_{\varphi_k(\mathbb{H})}\text{Log}^*\omega=\int_{\varphi_k(\mathbb{H})}\arg^*\omega_\theta=\left( \frac{2k+1}{m_1}-1\right)\pi^2.$$
In particular, the quantum index of $\psi_k$ is
$$k(\mathbb{H},\psi_k)=(m_3-m_2)(2k+1-m_1).$$
\end{lem}

\begin{figure}
\begin{center}
\definecolor{ffqqqq}{rgb}{1.,0.,0.}
\definecolor{qqqqff}{rgb}{0.,0.,1.}
\begin{tikzpicture}[line cap=round,line join=round,>=triangle 45,x=0.5cm,y=0.5cm]
\clip(-1.,-1.) rectangle (11.,11.);
\fill[color=qqqqff,fill=qqqqff,fill opacity=0.1] (0.,10.) -- (3.,10.) -- (3.,7.) -- cycle;
\fill[color=qqqqff,fill=qqqqff,fill opacity=0.1] (0.,0.) -- (3.,3.) -- (3.,0.) -- cycle;
\fill[color=ffqqqq,fill=ffqqqq,fill opacity=0.1] (3.,3.) -- (3.,7.) -- (5.,5.) -- cycle;
\fill[color=qqqqff,fill=qqqqff,fill opacity=0.1] (5.,5.) -- (7.,7.) -- (7.,3.) -- cycle;
\fill[color=ffqqqq,fill=ffqqqq,fill opacity=0.1] (7.,10.) -- (7.,7.) -- (10.,10.) -- cycle;
\fill[color=ffqqqq,fill=ffqqqq,fill opacity=0.1] (7.,3.) -- (7.,0.) -- (10.,0.) -- cycle;
\draw (0.,0.)-- (10.,0.);
\draw (10.,0.)-- (10.,10.);
\draw (10.,10.)-- (0.,10.);
\draw (0.,10.)-- (0.,0.);
\draw (0.,0.)-- (10.,10.);
\draw (0.,10.)-- (10.,0.);
\draw (7.,0.)-- (7.,10.);
\draw (3.,0.)-- (3.,10.);
\end{tikzpicture}

\caption{Co-amoeba of $\varphi_k$ with order map: $-1$ for red (triangles with left vertical side), $+1$ for blue (triangles with right vertical side).}
\label{coamoeba}
\end{center}
\end{figure}
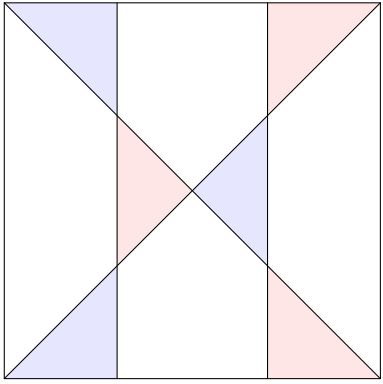

\begin{proof}
We compute the area of the coamoeba. According to \cite{forsgaard2015order}, the coamoeba with its order map is as on Figure \ref{coamoeba}. The order map has value $1$ on the blue triangles, and $-1$ on the red ones. The center point has coordinates $(0,0)$, the square has side length $2\pi$. The two vertical lines have respective abscissa $\pm\arg(i-c_k)=\pm\frac{2k+1}{2m_1}\pi$. This is the whole co-amoeba, which might fold itself, and we want to compute the area of half the co-amoeba, \textit{i.e.} the part corresponding to $S$. As $\arg(t-c_k)\in]0;\pi[$ if and only if $t\in \mathbb{H}$, we obtain $\arg\varphi_k(\mathbb{H})$ by restricting to the right half-square. Therefore, the area is obtained by taking the blue area minus the red area (because it comes with a minus sign). Let $l_k=\frac{2k+1}{2m_1}\pi$ the abscissa of the right vertical line. Then, we have
$$\mathcal{A}_{\arg} = l^2 - (\pi-l)^2 
	= 2\pi l -\pi^2 
	= \left( \frac{2k+1}{m_1}-1\right)\pi^2.
	$$
To get the quantum index of $\psi_k$, we multiply by the determinant of the monomial map, whose value is $m_1(m_3-m_2)$.
\end{proof}

\begin{rem}
It is easy to find the order map of the co-amoeba, which corresponds to the number of antecedents counted with sign. This function is constant on the complement of the \textit{shell}, which is a union of geodesics in the torus, directed by the vectors of $\Delta$. Moreover, the value of the order map changes by one when passing through one of the geodesic of the shell. This defines the order map up to a shift. The shift is fixed by the fact that the whole signed area is zero. One might then expect that adding the areas of some components of the complement of the shell would give the area of half the co-amoeba, \textit{i.e.} $\arg\varphi(\mathbb{H})$ instead of $\arg\varphi(\CC C)$. In general it is not the case: the order maps takes into account the antecedents on both connected components of $\CC C\backslash\RR C$, and it is not possible to draw them apart. However it is possible here since one of the monomials provides a coordinate on the curve, here $x$. Thus, we get $\arg\varphi(\mathbb{H})$ by restricting to half the argument torus, here the right-half square. 
\end{rem}

In particular, given two purely imaginary points on $\CC E_1$ and one real point on $\CC E_3$, we have proven that
\begin{itemize}[label=-]
\item If $m_1$ is odd, there exists precisely $m_1$ curves maximally tangent to the divisors and passing through the chosen points. Moreover, according to the two choices of orientation for each of them, there are two curves of each quantum index $(m_3-m_2)(2k+1-m_1)$, for $k\in[\![0;m_1-1]\!]$, \textit{i.e.} all the even multiples of $m_3-m_2$ of absolute value $<m_1$. This set is stable by one of the deck transformations. Thus, we get $2m_1$ oriented curves, two of each quantum index. If we also consider curves passing through the symmetric real point, we get $4m_1$ real oriented curves.
\item If $m_1$ is even, we might still be in the previous case (if $\frac{m_1}{\delta}$ is odd), or there might be $2m_1$ or $0$ solutions according to the sign of $\mu$ (when $\frac{m_1}{\delta}$ is even). Thus, there are $4m_1$ or $0$ oriented curves passing through the points.
\end{itemize}

The remaining cases of interest to compute quantum indices near the tropical limit are the case of a flat trivalent vertex, and a curve with one real puncture and two pairs of complex conjugated punctures. However, we do not need them in our computations.

\section{Tropical enumerative problem and refined curve counting}

	\label{tropical enumerative inv}

Let $\Delta\subset N$ be a family of $m$ primitive lattice vectors, with total sum $0$. As described in the introduction, there is an associated lattice polygon $P_\Delta$ having $m$ lattice points on its boundary. The toric surface obtained from $\Delta$ is denoted by $\CC\Delta$. Let $E_1,\dots, E_p$ denote the sides of the polygon $P_\Delta$ and let $n_1,\dots,n_p\in N$ be their normal primitive vectors. Let $s\leqslant\frac{l(E_1)}{2}$ be an integer, $r_1=l(E_1)-2s$, and $r_i=l(E_i)$ if $i\geqslant 2$, so that we have $\sum_1^p r_i +2s=m$. Let
$$\Delta(s)=(\Delta\backslash\{n_1^{2s}\})\cup\{(2n_1)^s\}=\{n_1^{r_1},(2n_1)^{s},n_2^{r_2},\dots,n_p^{r_p} \}.$$

\subsection{Tropical Problem}

The tropical curves of degree $\Delta(s)$ have $m-s$ unbounded ends, and therefore the moduli space $\mathcal{M}_0(\Delta(s), N_\RR)$ of parametrized rational tropical curves of degree $\Delta(s)$ in $N_\RR\simeq\RR^2$ has dimension $m-s-1$. We have the evaluation map:
$$\text{ev} : \mathcal{M}_0(\Delta(s), N_\RR)\longrightarrow \RR^{m-s-1},$$
that associates to a parametrized tropical curve the moments of every unbounded end but the first. Recall that the moment of the first unbounded end is equal to minus the sum of the other moments because of the tropical Menelaus theorem. Let $\mu\in\RR^{m-s-1}$ be a generic family of moments. We look for parametrized rational tropical curves $\Gamma$ of degree $\Delta(s)$ such that $\text{ev}(\Gamma)=\mu$.\\

Due to generality, as noticed in \cite{blomme2019caporaso}, every parametrized rational tropical curve $\Gamma$ such that $\text{ev}(\Gamma)=\mu$ is a simple nodal tropical curve and thus has a well-defined refined multiplicity. We then set
$$N_{\Delta(s)}^{\partial,\text{trop}}(\mu)=\sum_{\text{ev}(\Gamma)=\mu}m_\Gamma^q\in\ZZ[q^{\pm\frac{1}{2}}].$$

\begin{theo}\cite{blomme2019caporaso}
The value of $N_{\Delta(s)}^{\partial,\text{trop}}(\mu)$ is independent of $\mu$ provided that it is generic.
\end{theo}

\begin{rem}
These tropical considerations are easy to generalize for curves of any degree, in particular with non-primitive vectors. This tropical part thus also applies for the problem with complex imaginary points on any divisor. However, we restrict ourselves to degrees with primitive vectors for any toric divisor but one, 
due to a temporary lack of correspondence theorem in the more general setting.
\end{rem} 

\subsection{Classical problem}
\label{classical problem}

Keeping previous notations, let $\mathcal{P}$ be a configuration of $2m-s$ points on the toric boundary $\partial\CC\Delta$ such that:
\begin{itemize}[label=-]
\item $\mathcal{P}$ has $s$ pairs of conjugated purely imaginary points on $\CC E_1$, and $r_1$ pairs of opposite real points on $\RR E_1$,
\item for each $i\geqslant 2$, the configuration $\mathcal{P}$ has $r_i$ pairs of opposite real points on $\RR E_i$,
\item $\mathcal{P}$ is subject to the Menelaus condition. 
\item we assume that $\mathcal{P}$ is generic among such configurations.
\end{itemize} 

The fact that $\mathcal{P}$ is subject to the Menelaus condition is a slight misuse since there are not the right number of points. This is due to the fact that we consider curves passing only through one point of each pair of real opposite points. Here, \textit{Menelaus condition} means the following.

\begin{defi}
Let $\mathcal{P}$ be as above. For each pair of opposite real points $\{\pm p_j\}$, let $\pm\mu_j$ be their complex moments, and for each pair of opposite purely imaginary points $\{\pm p_j\}$, let $\pm i\lambda_j$ be their moments. The symmetric configuration $\mathcal{P}$ is said to satisfy the \textit{Menelaus condition} if
$$\prod_j |\mu_j| \prod_j \lambda_j^2=1.$$
\end{defi}

\begin{rem}
This means that the Menelaus condition is satisfied up to sign when we take both points of each non-real pair, and one point of each real pair.
\end{rem}

Let $\mathcal{S}(\mathcal{P})$ be the set of oriented real rational curves that pass through at least one point of each pair. Such a curve is said to \textit{pass through the symmetric configuration} $\mathcal{P}$. As the curves are oriented, each real curve is counted twice: once with each of its orientations. Notice that if a curve passes through one of the points of a pair of non-real points, it contains both since the curve is real. Every oriented curve of $\mathcal{S}(\mathcal{P})$ has real or purely imaginary intersection points, and thus a well-defined quantum index. We denote by $\mathcal{S}_k(\mathcal{P})$ the subset of $\mathcal{S}(\mathcal{P})$ formed by oriented curves with quantum index $k$.\\

Let $\varphi:\CC P^1\rightarrow\CC\Delta$ be an oriented real parametrized rational curve, denoted by $\overrightarrow{C}$. The logarithmic Gauss map sends a point $p\in\RR P^1$ to the tangent direction to $\text{Log}\varphi(\RR P^1)$ inside $N_\RR$. We get a map
$$\gamma:\RR P^1\rightarrow \PP^1(N_\RR).$$
The first space $\RR P^1$ is oriented since the curve is oriented, while $\PP^1(N_\RR)$ is oriented by $\omega$. The degree of this map is denoted by $\text{Rot}_\text{Log}(\overrightarrow{C})\in\ZZ$. If the curve has transverse intersections with the divisors, it has the same parity as the number of boundary points $m$. We then set
$$\sigma(\overrightarrow{C})=(-1)^\frac{m-\text{Rot}_\text{Log}(\overrightarrow{C})}{2}\in\{\pm 1\}.$$
Now let
$$R_{\Delta,k}(\mathcal{P)}=\sum_{\overrightarrow{C}\in\mathcal{S}_k(\mathcal{P})}\sigma(\overrightarrow{C}),$$
and
$$R_{\Delta}(\mathcal{P)}=\frac{1}{4}\sum_{k}R_{\Delta,k}(\mathcal{P)}q^k\in\ZZ [q^{\pm\frac{1}{2}}].$$
The coefficient $\frac{1}{4}$ is here to account for the deck transformation: if $\{f(x,y)=0\}$ is a curve in $\mathcal{S}(\mathcal{P})$, then $\{f(x,-y)=0\}$, $\{f(-x,y)=0\}$, $\{f(-x,-y)=0\}$ are in $\mathcal{S}(\mathcal{P})$ too.

\begin{theo}\cite{mikhalkin2017quantum}
The value of $R_{\Delta}(\mathcal{P)}$ is independent of the configuration $\mathcal{P}$ as long as it is generic.
\end{theo}

The obtained polynomial, independent of $\mathcal{P}$, is denoted by $R_{\Delta,s}$.

\begin{rem}
This is in fact a specialization of the theorem of \cite{mikhalkin2017quantum} presented in the introduction, with pairs of purely imaginary points located on a common divisor. That is why we adopt the notation $R_{\Delta,s}$ rather than $R_{\Delta,(s,0,\dots,0)}$.
\end{rem}

\section{Realization and correspondence theorem in the real case}

In this section we prove a correspondence theorem, by refining the realization theorem of Tyomkin \cite{tyomkin2017enumeration} in the case of real curves.
The proof follows the same steps as in \cite{tyomkin2017enumeration} and presents similar calculations. We start by giving a bunch of notations which might seem a little heavy, but are useful to deal with real tropical curves having a non-trivial real structure.\\ 

	\subsection{Notations}

Let $\Gamma$ be a real rational abstract  tropical curve with $m$ ends. We denote by $\sigma$ the involution on $\Gamma$, and $\Gamma/\sigma$ the quotient graph. When needed, we denote by $\pi:\Gamma\rightarrow\Gamma/\sigma$ the quotient map. 
Let $I$ denote the set of ends of $\Gamma$, endowed with an action of the involution, embodied in the following decomposition:
$$I=\{ x_1,\dots,x_r,z_1^\pm,\dots,z_s^\pm\},$$
where the ends $x_i$ are fixed ends (real ends) and $z_i^\pm$ are exchanged with one another (complex ends). Following this notation, we denote the set of ends of $\Gamma/\sigma$ by:
$$I/\sigma=\{ x_1,\dots,x_r,z_1,\dots,z_s\}.$$
We assume that $r\geqslant 1$, and orient the edges of both $\Gamma$ and $\Gamma/\sigma$ away from $x_r$, which makes them rooted trees. This orientation induces a partial order $\prec$ on the curve. For $w$ and $w'$ vertices or ends, we have $w\prec w'$ if and only if the shortest path from $w$ to $w'$ agrees with the orientation of the graph. We endow the set $I/\sigma$ with a total order, different from $\prec$, for which $x_r$ is the smallest element.\\

If $w\in\Gamma^0$ is a vertex of $\Gamma$, let $I_w^\infty$ be the set of ends of $\Gamma$ which are greater than $w$ for the order $\prec$.We take a similar notation $(I/\sigma)^\infty_w$ for $w\in(\Gamma/\sigma)^0$. Notice that if $w\in\fix$, then $I^\infty_w$ is stable by $\sigma$, and if $w\notin\fix$, then at most one element of each pair $\{z_j^\pm\}$ belongs to $I^\infty_w$.\\

If $\gamma\in\Gamma^1$ is a bounded edge of $\gamma$ (same for $\Gamma/\sigma$), let $\tg$ and $\hg$ be the tail and the head of $\gamma$. Notice that $\hg\notin\fix$ if and only if $\gamma\notin\fix$.\\

If $\gs\in(\Gamma/\sigma)^1$ is an edge of $\Gamma/\sigma$, let $\iota(\gamma^\sigma)$ be the smallest element of $(I/\sigma)^\infty_\hgs$. It is the smallest end among those accessible by $\hgs$. The order on $I/\sigma$ along with this map $\iota$ induces an order on the edges of $\Gamma/\sigma$ having the same tail. We thus can speak about the smallest and biggest edge leaving a vertex $\pi(w)$ of $\Gamma/\sigma$. We can then lift these local orders on $\Gamma/\sigma$ to $\Gamma$:
	\begin{itemize}[label=-]
	\item If $w\in\fix$, then we have three cases for the lift of an edge $\gs$ such that $\tgs=\pi(w)$:
		\begin{itemize}[label=$\star$]
		\item $(RR)$ the edge $\gs=\{\gamma\}$ lifts to a fixed edge $\gamma$ of $\Gamma$, and $\iota(\gs)=\{x_j\}$ is a real marking. (The $(RR)$ stands for "Real edge Real marking".) We then set $\iota(\gamma)=x_j$.
		\item $(RC)$ the edge $\gs=\{\gamma\}$ lifts to a fixed edge $\gamma$ of $\Gamma$ but $\iota(\gs)=\{z_j^\pm\}$ is a complex marking. It means that the curve $\Gamma$ splits at some point on the path from $w$ to $\iota(\gs)$, but not right away in $\gamma$. (The $(RC)$ stands for "Real edge Complex marking".)
		\item $(CC)$ the edge $\gs=\{\gamma^+,\gamma^-\}$ lifts to a pair of exchanged edges in $\Gamma$. If $\iota(\gs)=\{z_j^\pm\}$, we have up to a relabeling of $\gamma^\pm$ that $z_j^+$ (resp. $z_j^-$) is accessible via $\gamma^+$ (resp. $\gamma^-$). (The $(CC)$ stands for "Complex edge Complex marking".)
		\end{itemize}
	\item If $w\notin\fix$, then $\pi:\{\gamma\in\Gamma^1:\tg=w\}\rightarrow\{\gs:\tgs=\pi(w)\}$ is a bijection, therefore we have a total order on $\{\gamma\in\Gamma^1:\tg=w\}$, and for every edge $\gamma$ such that $\tg=w$ we have a unique complex end $\iota(\gamma)\in\iota(\gs)$ accessible by $w$. Notice that in this case, every edge $\gs$ such that $\tgs=\pi(w)$ is of type $(CC)$. Moreover, we also have an induced order between the edges emanating from $w$.
	\end{itemize}
	
	We denote by $v_\RR$ (resp. $v_\CC$) the number of fixed vertices (resp. pairs of exchanged vertices), and by $e_\RR$ (resp. $e_\CC$) the number of fixed bounded edges (resp. pairs of exchanged bounded edges).

	\subsection{Space of rational curves with given tropicalization} 

Let $\Gamma$ be an abstract tropical curve with $m$ ends. Let $(C^{(t)},x_1,\dots,x_r,z_1^\pm,\dots,z_s^\pm)$ be a real smooth rational curve with a real configuration $(\textbf{x},\textbf{z}^\pm)=(x_1,\dots,x_r,z_1^\pm,\dots,z_s^\pm)$ of marked points, tropicalizing on $\Gamma$. We assume that $C^{(0)}$ is the special fiber of the stable model of $C^{(t)}$, so that $\Gamma$ is the dual graph of $C^{(0)}$. We assume that $r\geqslant 1$ and we also denote by $\sigma$ the real structure on $C^{(t)}$.\\

\begin{rem}
By taking a coordinate, the marked curve $(C^{(t)},\textbf{x},\textbf{z}^\pm)$ can be seen as $\PP^1\left(\CC((t))\right)\simeq\CC((t))\cup\{\infty\}$, the projective line over the field of Laurent series, along with $r+2s$ Laurent series which are the marked points, taken up to a change of coordinate in $GL_2\big(\RR((t))\big)$. The first $r$ Laurent series are in $\RR((t))$, and the last $s$ are taken in $\CC((t))\backslash\RR((t))$ along with their conjugate.
\end{rem}

We associate to each vertex $w\in\Gamma^0$ a coordinate $y_w$ on $C^{(t)}$, taking into account the real structure, \textit{i.e.} the coordinate $y_w$ is real if $w\in\fix$ and $y_{\sigma(w)}=\overline{y_w\circ\sigma}$ otherwise. Moreover, the coordinate $y_w$ specializes to a coordinate on the irreducible component of $C^{(0)}$ associated to $w$.

\begin{itemize}[label=-]
\item If $w\in\fix$, the set $I^\infty_w$ is stable by $\sigma$, but has no order induced by $I/\sigma$ since for each complex marking, both are accessible. Still, let $\gamma^\sigma_a$ and $\gamma^\sigma_b$ be the smallest and biggest edges emanating from $\pi(w)$ in $\Gamma/\sigma$. We make a disjunction according to the type $(RR),(RC),(CC)$ of each edge:
	\begin{itemize}[label=$\star$]
	\item[$(RR/RR)$] If $\gamma^\sigma_a$ and $\gamma^\sigma_b$ are both of type $(RR)$, they lift to edges $\gamma_a$ and $\gamma_b$ of $\Gamma$, which have well-defined real marking $x_a=\iota(\gamma_a)$ and $x_b=\iota(\gamma_b)$. Then we take $y_w$ such that $y_w(x_r)=\infty$, $y_w(x_a)=0$, $y_w(x_b)=1$. This is a real coordinate since $y_w$ coincide with $\overline{y_w\circ\sigma}$ at three points.
	\item[$(RR/RC)$] If $\gamma^\sigma_a$ is of type $(RR)$ and $\gamma^\sigma_b$ of type $(RC)$, then they lift up to edges $\gamma_a,\gamma_b\in\fix$, and we have $\iota(\gamma_a)=x_a$, and $\iota(\gamma_b)=z_b^\pm$. Then $\re z_b^\pm$ is a well-defined real Laurent series whose specialization on the component associated to $w$ is different from the one of $x_a$. Then we can take $y_w$ such that $y_w(x_r)=\infty$, $y_w(x_a)=0$, $y_w(\re z_b^\pm)=1$.
	\item[$(RC/RR)$] We do the same with $\re z_a^\pm$ and $x_b$.
	\item[$(RC/RC)$] If $\gamma^\sigma_a$ and $\gamma^\sigma_b$ are both of type $(RC)$ we do the same with $\re z_a^\pm$ and $\re z_b^\pm$.
	\item[$(CC/-)$] If $\gamma^\sigma_a$ is of type $(CC)$, then $\gamma^\sigma_a$ lifts to a pair of exchanged edges $\{\gamma_a^\pm\}$ both emanating from $w$. They both have a well-defined $\iota(\gamma_a^\pm)=z_a^\pm$. Then we take $y_w$ such that $y_w(x_r)=\infty$, $y_w(z_a^\pm)=\pm i$, which also is a real coordinate.
	\item[$(-/CC)$] If $\gamma^\sigma_a$ is of type $(RR)$ or $(RC)$ and $\gamma^\sigma_b$ is of type $(CC)$, we do the same with $z_b^\pm$.
	\end{itemize}
\item If $w\notin\fix$, then $I_w^\infty$ consists only of complex markings, all edges emanating from $w$ are of type $(CC)$ and we have a well-defined $\iota(\gamma)$ for each of them. Let $a$ and $b$ be the smallest and biggest elements in $I_w^\infty$. We take $y_w$ such that $y_w(x_r)=\infty$, $y_w(a)=0$, $y_w(b)=1$. This choice ensures that $y_{\sigma(w)}=\overline{y_w\circ\sigma}$.
\end{itemize}

The functions $y_w$ are all coordinates on $C$ sending $x_r$ to $\infty$, therefore we can pass from one to another by a real affine function which we now describe.

\begin{prop}
Let $\gamma\in\Gamma^1$ be a bounded edge.
\begin{itemize}[label=-]
\item If $\gamma\notin\fix$, let $z_a^\varepsilon$ and $z_b^\eta$ be the smallest and biggest elements in $I^\infty_\hg$, then
	$$y_\hg=\frac{y_\tg - y_\tg(z_a^\varepsilon)}{y_\tg(z_b^\eta) - y_\tg(z_a^\varepsilon)} \text{ and }|\gamma|=\val(y_\tg(z_b^\eta) - y_\tg(z_a^\varepsilon)).$$
\item If $\gamma\in\fix$, we make a disjunction according to the type of $\hg$:
	\begin{itemize}[label=$\star$]
	\item[$(RR/RR)$] $y_\hg=\frac{y_\tg - y_\tg(x_a)}{y_\tg(x_b) - y_\tg(x_a)} \text{ and }|\gamma|=\val(y_\tg(x_b) - y_\tg(x_a)).$
	\item[$(RR/RC)$] $y_\hg=\frac{y_\tg - y_\tg(x_a)}{\re y_\tg(z_b^\pm) - y_\tg(x_a)} \text{ and }|\gamma|=\val(\re y_\tg(z_b^\pm) - y_\tg(x_a)).$
	\item[$(RC/RR)$] $y_\hg=\frac{y_\tg - \re y_\tg(x_a)}{y_\tg(x_b) - \re y_\tg(z_a^\pm)} \text{ and }|\gamma|=\val(y_\tg(x_b) - \re y_\tg(z_a^\pm)).$
	\item[$(RC/RC)$] $y_\hg=\frac{y_\tg - \re y_\tg(z_a^\pm)}{\re y_\tg(z_b^\pm) - \re y_\tg(z_a^\pm)} \text{ and }|\gamma|=\val(\re y_\tg(z_b^\pm) - \re y_\tg(z_a^\pm)).$
	\item[$(CC/-)$] $y_\hg=\frac{y_\tg-\re y_\tg(z_a^\pm)}{\im y_\tg(z_a^+)} \text{ and } |\gamma|=\val\left(\im y_\tg(z_a^+)\right).$
	\item[$(-/CC)$] same with $a$ switched by $b$.
	\end{itemize}
\end{itemize}
\end{prop}

\begin{proof}
In each case we check that the right-hand term, which is a coordinate since it is obtained by an affine change from another coordinate, coincides with $y_\hg$ at the three points used to define it. As it is so, they are equal. The equality with the length of $\gamma$ is the definition of the latter since $y_\hg$ and $y_\tg$ are coordinates on the irreducible component associated with $\hg$ and $\tg$.
\end{proof}

For every edge we now define $\alpha_\gamma\in\CC[[t]]^\times$ and $\beta_\gamma\in\CC[[t]]$ which will be the coordinates on the space of real marked curves tropicalizing on $\Gamma$. Once again, the  definition goes through the distinction of the type of edges emanating from $\hg$. Let $\gamma\in\Gamma^1_b$ be a bounded edge:
	\begin{align*}
	(RR/RR)\ & \alpha_\gamma=t^{-|\gamma|}\left(y_\tg(x_b) - y_\tg(x_a)\right), \\
	(RR/RC)\  & \alpha_\gamma=t^{-|\gamma|}\left(\re y_\tg(z_b^\pm) - y_\tg(x_a)\right), \\
	(RC/RR)\  & \alpha_\gamma=t^{-|\gamma|}\left(y_\tg(x_b) - \re y_\tg(z_a^\pm)\right), \\
	(RC/RC)\  & \alpha_\gamma=t^{-|\gamma|}\left(\re y_\tg(z_b^\pm) - \re y_\tg(z_a^\pm)\right), \\
	(CC/-)\  & \alpha_\gamma=t^{-|\gamma|}\im y_\tg(z_a^+), \\
	(-/CC)\  & \alpha_\gamma=t^{-|\gamma|}\im y_\tg(z_b^+). \\
	\end{align*}
	Let $\gamma\in\Gamma^1$ be a non-necessarily bounded edge:
	\begin{itemize}[label=$\star$]
	\item If $\gamma\notin\fix$ is of type $(CC)$ then $\beta_\gamma=y_\tg(z_{\iota(\gamma)}^\varepsilon)$ where $z_{\iota(\gamma)}^\varepsilon$ is the lift of $z_{\iota(\gamma)}$ accessible by $\gamma$.
	\item If $\gamma\in\fix$ is of type $(RR)$ then $\beta_\gamma=y_\tg(x_{\iota(\gamma)})$.
	\item If $\gamma\in\fix$ is of type $(RC)$ then $\beta_\gamma=\re y_\tg(z_{\iota(\gamma)}^\pm)$.
	\end{itemize}
	
	We now can define the function
	$$\Psi_\gamma(y)=\beta_\gamma + t^{|\gamma|}\alpha_\gamma y,$$
	which allows an easy description of the relations between the $y_w$.
	
	\begin{prop}\label{propcalcul}
	The Laurent series $\alpha_\gamma$ and $\beta_\gamma$ satisfy the following properties.
	\begin{enumerate}[label=(\roman*)]
	\item If $\gamma$ is an edge, one has $\alpha_\gamma\in\CC[[t]]^\times$. Moreover, if $\gamma\neq\gamma'$ are two different edges with the same tail $\tg=\mathfrak{t}(\gamma')$, then $\beta_\gamma-\beta_{\gamma'}\in\CC[[t]]^\times$.
	\item They are real: $\alpha_{\sigma(\gamma)}=\overline{\alpha_\gamma}$, $\beta_{\sigma(\gamma)}=\overline{\beta_\gamma}$. In particular $\alpha_\gamma,\beta_\gamma\in\RR[[t]]$ if $\gamma\in\fix$.
	\item For each edge $\gamma$, one has $y_\tg=\Psi_\gamma(y_\hg)$. In particular, for any marked point $q$, one has
	$$y_\tg-y_\tg(q)=t^{|\gamma|}\alpha_\gamma\left(y_\hg-y_\hg(q)\right).$$
	Moreover $\Psi_\gamma$ is real in the sense that $\Psi_{\sigma(\gamma)}(y)=\overline{\Psi_\gamma(\overline{y})}$.
	\item Let $w,w'$ be two vertices, and $\gamma_1,\dots,\gamma_d$ the geodesic path between from $w$ to $w'$. Let $\varepsilon_i=\pm 1$ according to the orientation of $\gamma_i$ in $\Gamma$, and the orientation in the geodesic path agree or disagree. Then
	$$y_w=\Psi_{\gamma_1}^{\varepsilon_1}\circ\cdots\circ\Psi_{\gamma_d}^{\varepsilon_d}(y_{w'}),$$
	and in particular for any marked point $q$:
	$$y_w-y_w(q)=t^{\sum_1^d \varepsilon_i|\gamma_i|}\left(\prod_1^d \alpha_{\gamma_i}^{\varepsilon_i}\right)\left(y_{w'}-y_{w'}(q)\right).$$
	\item Let $q$ be a marked point associated with an unbounded end $e\in\Gamma_\infty^1$, $w$ be a vertex, and $\gamma_1,\dots,\gamma_d,e$ be the geodesic path from $w$ to $q$. Then
	$$y_w(q)=\Psi_{\gamma_1}^{\varepsilon_1}\circ\cdots\circ\Psi_{\gamma_d}^{\varepsilon_d}(\beta_e),$$
	and in particular for every marked point $q$ associated with the unbounded end $e$, if $v_r$ is the vertex adjacent to the unbounded end associated to $x_r$,
	$$y_{v_r}(q)=\beta_{\gamma_1}+t^{|\gamma_1|}\alpha_{\gamma_1}\left( \cdots (\beta_{\gamma_d}+t^{|\gamma_d|}\alpha_{\gamma_d}\beta_e \right).$$
	\item For every marked point $q_i$ and vertex $w\in\Gamma^0$, $\val(y_w(q))\geqslant 0\Leftrightarrow$ $q$ is accessible by $w$.
	\item For every edge $\gamma\in\Gamma^1$ and every marked point $q\in I^\infty_\tg\backslash I_\hg^\infty$, we have $\val(y_\tg(q)-\beta_\gamma)=0$.
	\end{enumerate}
	\end{prop}
	
	\begin{proof}
	It suffices to check every statement:
	\begin{itemize}[label=-]
	\item $(i)$ and $(ii)$ follow from the definition of $\alpha$ and $\beta$,
	\item $(iii)$ comes from the definition of $\Psi_\gamma$ and from the fact that $\alpha$ and $\beta$ are real,
	\item $(iv)$ and $(v)$ are just iterations from $(iii)$,
	\item $(vi)$ comes from $(v)$ and,$(vii)$ follows from $(i)$ and $(v)$.
	\end{itemize}	 
	\end{proof}
	
	\begin{rem}
	Formally speaking, this proposition is the direct translation of Proposition $4.3$ in \cite{tyomkin2017enumeration}, in the setting of curves with a real structure. Although the formulas seem quite repulsive at the first look, the meaning of each object must be clear. 
The formal series $\alpha_\gamma$ and $\beta_\gamma$ allow one to recover the coordinates of the marked points, following the formula $(v)$ of Proposition \ref{propcalcul}. 
The formal series $\alpha_\gamma$ are the \textit{"phase length"} of the edge $\gamma$, in contrast to $t^{|\gamma|}$ which could be called \textit{"valuation length"}, while the formal series $\beta_\gamma$ are the directions one needs to follow at each $w$ in order to get to the points of $I^\infty_\hg$. In other terms, $\alpha$ and $\beta$ provide the necessary coefficients to find the coordinates of the marked points. One could say that the abstract tropical curve $\Gamma$ only remembers the valuation information, while the formal series $\alpha$ and $\beta$ encode the phase information. 
	\end{rem}
	
	Let $v=v_r$ be the vertex adjacent to the end $x_r$, the uplet
	$$(y_v(x_1),\dots,y_v(x_{r-1}),y_v(z_1^+),\dots,y_v(z_s^+))\in\RR((t))^{r-1}\times\CC((t))^s$$
	provides a system of coordinates on the moduli space of real rational marked curves, that we can restrict on the moduli space of curves tropicalizing on $\Gamma$. Notice that the choice of $y_v$ fixes the value of some members of the uplet. The definition of $\alpha,\beta$ along with a quick induction ensures that they can be written in terms of $\big(y_v(q)\big)_q$. Conversely, the formula from Proposition \ref{propcalcul}$(v)$ allows to recover $\big(y_v(q)\big)_q$ from $\alpha$ and $\beta$. Therefore, they also provide a system of coordinates. Moreover,   Proposition \ref{propcalcul} describes the set of possible values of $\alpha$, $\beta$, since the formula from Proposition \ref{propcalcul}$(v)$ gives the values of the points to choose on $\PP^1\left(\CC((t))\right)$, in order to make it into a marked curve with the right tropicalization and the right formal series $\alpha$, $\beta$.\\
	
	We denote by $\mathcal{A}$ the space $\left(\RR[[t]]^\times\right)^{e_\RR} \times \left(\CC[[t]]^\times\right)^{e_\CC}$ of possible values of $\alpha$. We denote by $\mathcal{B}$ the space of possible values of $\beta$ satisfying the conditions of Proposition \ref{propcalcul}.

	\subsection{Space of morphisms with given tropicalization} 
	
	Let $h:\Gamma\rightarrow N_\RR$ be a real rational parametrized tropical curve of degree $\Delta$. In this subsection we give an explicit description of morphisms $f:(C,\textbf{x},\textbf{z}^\pm)\rightarrow\text{Hom}(M,\CC((t))^*)$ of degree $\Delta$ tropicalizing to it, and for which $(C,\textbf{x},\textbf{z}^\pm)$ is a smooth connected marked rational curve.\\
	
	Let $f:C\rightarrow\text{Hom}(M,\CC((t))^*)$ be a real morphism that tropicalizes to $h$. By assumption, in the coordinate $y_w$, the morphism $f$ takes the following form:
	$$f(y_w)=t^{h(w)}\chi_w\prod_{i\in I^\infty_w} (y_w-y_w(q_i))^{n_i} \prod_{i\notin I^\infty_w}\left( \frac{y_w}{y_w(q_i)}-1\right)^{n_i} \in\text{Hom}(M,\CC((t))^*).$$
	
\begin{rem}	
Notice that $t^{h(w)}$ denotes the morphism $m\mapsto t^{h(w)(m)}\in\CC((t))^*$. The choice of normalization in the product ensures that $\chi_w:M\rightarrow\CC((t))^*$ has value in $\CC[[t]]^\times$. Finally, both products are indexed by the ends of $\Gamma$, and although the writing does not emphasize this aspect, there are real ends and complex ends. Furthermore, there is a constant term $(-1)$ in the second product, corresponding to the root $x_r$, for which $y_w(x_r)=\infty$ for any $w$.
\end{rem}

We now relate the expression of $f$ in two different vertices. We assume that they are connected by an edge $\gamma$. Let
$$\phi_\gamma=\prod_{i\in I^\infty_\tg\backslash I^\infty_\hg}\left( y_\tg(q_i)-\beta_\gamma\right)^{n_i} \prod_{i\notin I^\infty_\tg}\left( 1-\frac{\beta_\gamma}{y_\tg(q_i)} \right)^{n_i}\in\text{Hom}(M,\CC((t))^*.$$
Notice that $\phi_\gamma$ depends only on the value of $\alpha,\beta$, and is thus a function $\phi_\gamma(\alpha,\beta)$.

\begin{prop}\label{morphism}
We have the following properties of $\chi_w$ and $\phi_\gamma$:
\begin{enumerate}[label=(\roman*)]
\item The $\chi_w$ are real: $\chi_{\sigma(w)}=\overline{\chi_w}$. In particular, if $w\in\fix$, $\chi_w$ takes values in $\RR[[t]]^\times$.
\item The co-characters $\phi_\gamma$ are real: $\phi_{\sigma(\gamma)}=\overline{\phi_\gamma}$. In particular, if $\gamma\in\fix$ is a fixed edge, the co-character $\phi_\gamma$ takes values in $\RR[[t]]^\times$.
\item For any edge $\gamma$, let $n_\gamma$ denote the slope of $h$ on $\gamma$. One has
$$\phi_\gamma\cdot\frac{\chi_\tg}{\chi_\hg}\cdot\alpha_\gamma^{n_\gamma}=1\in\text{Hom}(M,\CC((t))^*).$$
\end{enumerate}
\end{prop}

\begin{proof}
The first two points are immediate to check and follow from the definition of $\phi_\gamma$ along with the fact that $f$ is real. For the last point, we start by making the quotient of the two expressions of $f$ in the coordinates $y_\tg$ and $y_\hg$, and use Proposition \ref{propcalcul} that ensures that $y_\tg-y_\tg(q_i)=t^{|\gamma|}\alpha_\gamma(y_\hg-y_\hg(q_i))$. Hence,
$$\frac{t^{h(\tg)}}{t^{h(\hg)}}\cdot\frac{\chi_\tg}{\chi_\hg}\cdot\frac{\prod_{i\notin I^\infty_\hg} y_\hg(q_i)^{n_i}}{\prod_{i\notin I^\infty_\tg} y_\tg(q_i)^{n_i}}\cdot \left(t^{|\gamma|}\alpha_\gamma\right)^{\sum_i n_i}=1\in\text{Hom}(M,\CC((t))^*).$$
Since $h(\hg)-h(\tg)=|\gamma|n_\gamma$, and $\sum_i n_i=0$ by balancing condition, we get
$$t^{-|\gamma|n_\gamma}\frac{\prod_{i\notin I^\infty_\hg} y_\hg(q_i)^{n_i}}{\prod_{i\notin I^\infty_\tg} y_\tg(q_i)^{n_i}}\cdot\frac{\chi_\tg}{\chi_\hg}=1\in\text{Hom}(M,\CC((t))^*).$$
Finally, using that $y_\hg(q_i)=\frac{y_\tg(q_i)-\beta_\gamma}{t^{|\gamma|}\alpha_\gamma}$, we get that
$$\frac{\prod_{i\notin I^\infty_\hg} y_\hg(q_i)^{n_i}}{\prod_{i\notin I^\infty_\tg} y_\tg(q_i)^{n_i}}    =  (t^{|\gamma|}\alpha_\gamma)^{-\sum_{i\notin I^\infty_\hg}n_i} \frac{\prod_{i\notin I^\infty_\hg} (y_\tg(q_i)-\beta_\gamma)^{n_i}}{\prod_{i\notin I^\infty_\tg} y_\tg(q_i)^{n_i}}.   $$
Since by adding the balancing condition at the vertices not accessible via $\gamma$ we get $\sum_{i\notin I^\infty_\hg}n_i=-n_\gamma$, and as $\complement I^\infty_\tg\subset\complement I^\infty_\hg$, we get that
$$\frac{\prod_{i\notin I^\infty_\hg} y_\hg(q_i)^{n_i}}{\prod_{i\notin I^\infty_\tg} y_\tg(q_i)^{n_i}} = t^{|\gamma|n_\gamma}\cdot\alpha_\gamma^{n_\gamma}\cdot\phi_\gamma,$$
which results in the desired formula.
\end{proof}
	
\begin{rem}
There is a slight misnomer in the proof: because $y_w(x_r)=\infty$, the intermediate steps of computation are not well-defined. However, the computation remains true, either by allowing a finite value to $y_w(x_r)$ and then making it $\infty$, or by putting it apart in the computation, which complicates the explanation.
\end{rem}

Conversely, if we are given a real family of $\chi_w:M\rightarrow\CC[[t]]^\times$ such that Proposition \ref{morphism} holds for any $\gamma$, then the maps defined by the formulas 
$$f(y_w)=t^{h(w)}\chi_w\prod_{i\in I^\infty_w} (y_w-y_w(q_i))^{n_i} \prod_{i\notin I^\infty_w}\left( \frac{y_w}{y_w(q_i)}-1\right)^{n_i}$$
agree and define a real morphism $f$ tropicalizing to $h:\Gamma\rightarrow N_\RR$.\\

If $G$ is an abelian group, let $N_G=N\otimes G$. Let $\mathcal{X}=N_{\RR[[t]]^\times}^{v_\RR}\times N_{\CC[[t]]^\times}^{v_\CC}$ be the space where the tuple $\chi$ is chosen. Then, the space of morphisms $f:(C,\textbf{x},\textbf{z}^\pm)\rightarrow N\otimes\CC((t))^*$ tropicalizing to $h:\Gamma\rightarrow N_\RR $ is the subset of $\mathcal{X}\times\mathcal{A}\times\mathcal{B}$ given by the following equations:
$$\forall\gamma\in \Gamma^1_b\ :\ \phi_\gamma(\alpha,\beta)\cdot\frac{\chi_\tg}{\chi_\hg}\cdot\alpha_\gamma^{n_\gamma}=1\in\text{Hom}(M,\CC((t))^*).$$
The tuples $\alpha,\beta$ deal with the tropicalization of the curve, and the tuple $\chi$ with the tropicalization of the morphism.
	
	\subsection{Evaluation map} 

We now use the previous description to write down the conditions that a curve having fixed moments must satisfy. For each $n_j\in\Delta$, let $m_j=\iota_{n_j}\omega$ be the monomial used to measure the moment of the corresponding end $q_j$.\\

Let $q_j$ be a real or complex marked point, and $v_j$ be the adjacent vertex of the associated unbounded end $e_j$. In the coordinate $y_{v_j}$, the expression of the moment of the marked point $q_j$ takes the following form:
$$f^*\chi^{m_j}|_{q_j} = t^{h(v_j)(m_j)}\chi_{v_j}(m_j)\prod_{i\in I^\infty_{v_j}}\left(\beta_{e_j}-y_{v_j}(q_i)\right)^{\omega(n_j,n_i)}\prod_{i\notin I^\infty_{v_j}}\left( \frac{\beta_{e_j}}{y_{v_j}(q_i)}-1\right)^{\omega(n_j,n_i)}.$$
We then put 
$$\varphi_j=\prod_{i\in I^\infty_{v_j}}\left(\beta_{e_j}-y_{v_j}(q_i)\right)^{\omega(n_j,n_i)}\prod_{i\notin I^\infty_{v_j}}\left( \frac{\beta_{e_j}}{y_{v_j}(q_i)}-1\right)^{\omega(n_j,n_i)}\in\CC[[t]]^\times,$$
which, according to Proposition \ref{propcalcul}, is an invertible formal series only depending on $\alpha,\beta$. The series is invertible since the only terms of the product having positive valuation are taken with a zero exponent.

\begin{prop}
The formal series $\varphi_j$ are real: $\varphi_{\sigma(j)}=\overline{\varphi_j}$ for every end $e_j$, and in particular, if $x_j$ is a real marked point, then $\varphi_j\in\RR[[t]]^\times$. Moreover, $\varphi$ is a function of $\alpha,\beta$, \textit{i.e.} it does not depend on $\chi$.
\end{prop}

\begin{proof}
It follows from the fact that all the quantities that intervene in the definition of $\varphi_j$ are real. The second part is obvious.
\end{proof}



Thus, with this new notation, the evaluation map takes the following form:

$$f^*\chi^{m_j}|_{q_j} = t^{h(v_j)(m_j)}\chi_{v_j}(m_j)\varphi_j.$$

	\subsection{Correspondence theorem}
	\label{correspondence theorem section}
	Now we look at the following map: 
	$$\begin{array}{rccl}
	\Theta : & \mathcal{X}\times\mathcal{A}\times\mathcal{B} &\longrightarrow & \left(N_{\RR[[t]]^\times}^{e_\RR}\times N_{\CC[[t]]^\times}^{e_\CC}\right)\times\RR[[t]]^{\times r-1}\times\CC[[t]]^{\times s} \\
		& (\chi,\alpha,\beta) & \longmapsto & \left( \left(\phi_\gamma\cdot\frac{\chi_\tg}{\chi_\hg}\cdot\alpha_\gamma^{n_\gamma}\right)_{\gamma\in\Gamma^1_b},\left(\chi_{v_j}(m_j)\varphi_j \right)_j \right) \\
	\end{array}.$$
	This map is the same as in \cite{tyomkin2017enumeration} but for a curve endowed with a non-trivial real involution. That is why we take every real vertex or real edge with real coefficients, and only one of every pair of complex vertices or complex edges with complex coefficients.\\
	
	One can check that the dimensions of the source and target spaces are the same:
	\begin{itemize}[label=-]
	\item Since $N$ has rank $2$, the target space has dimension $2e_\RR+4e_\CC+r-1+2s$.
	\item The space $\mathcal{A}$ has dimension $e_\RR+2e_\CC$ since one chooses a real formal series for each real edge, and a complex one for each pair of conjugated edges.
	\item The space $\mathcal{X}$ has dimension $2(v_\RR+2v_\CC)$ since $N$ has rank $2$.
	\item The dimension of $\mathcal{B}$ is precisely $\text{ov}(\Gamma)$ since we choose a real or complex coefficient for each edge $\gamma$ for which $\tg$ is not trivalent.
	\end{itemize}
	We thus have to check that
	$$2e_\RR+4e_\CC+r+2s-1=e_\RR+2e_\CC+2(v_\RR+2v_\CC)+\text{ov}(\Gamma),$$
which is equivalent to 
	$$\text{ov}(\Gamma)-(r+2s)=e_\RR+2e_\CC-2(v_\RR+2v_\CC)-1 .$$
	Since $\Gamma$ is a tree, the Euler characteristic gives the following relation:
	$$1+e_\RR+2e_\CC=v_\RR+2v_\CC.$$
	The count of the valencies of the vertices gives also:
	$$r+2s+2e_\RR+4e_\CC=3(v_\RR+2v_\CC)+\text{ov}(\Gamma).$$
	These relations lead to
	\begin{align*}
	\text{ov}(\Gamma)-(r+2s) &=2e_\RR+4e_\CC-3v_\RR-6v_\CC \\
	&=e_\RR+2e_\CC-2(v_\RR+2v_\CC)-1. \\
	\end{align*}
	Thus, the dimensions are equal.\\
	
	Let $\zeta\in\RR((t))^{*r-1}\times i\RR((t))^{*s}$ be a generic family of moments, defining a real configuration of points $\mathcal{P}_0$, and thus a symmetric configuration of points $\mathcal{P}$ in $\CC\Delta$ when considering possible changes of signs. Let $\mu_j=\val\zeta_j\in\RR$ be their respective "tropical" moments, which we also assume to be generic. We denote by $\zeta_j^\Gamma=\zeta_j t^{-\mu_j}\in\CC[[t]]^\times$ the moments normalized to have a zero valuation. Any classical curve passing through the symmetric configuration $\mathcal{P}$ tropicalizes on a real parametrized tropical curve $h:\Gamma\rightarrow N_\RR$ satisfying $\text{ev}(\Gamma)=\mu$, and specializes to a tuple $(\chi,\alpha,\beta)$ in $\mathcal{X}\times\mathcal{A}\times\mathcal{B}$, satisfying $\Theta(\chi,\alpha,\beta)=(1,\zeta^\Gamma)$. Moreover, the plane tropical curve image $h(\Gamma)$ has a unique parametrization as a parametrized tropical curve of degree $\Delta(s)$. Notice that the space $\mathcal{X}\times\mathcal{A}\times\mathcal{B}$ depends on the choice of the parametrized tropical curve $(\Gamma,h)$.\\
	
	Conversely, for each real parametrized tropical curve $(\Gamma,h)$ with $\text{ev}(\Gamma)=\mu$, we need to find the classical curves tropicalizing on $(\Gamma,h)$ and passing through the symmetric configuration $\mathcal{P}$. Such a curve corresponds to a point in the moduli space $\mathcal{X}\times\mathcal{A}\times\mathcal{B}$. Finding the curves passing through the symmetric configuration $\mathcal{P}$ and tropicalizing on $\Gamma$ thus amounts to solve for $(\chi,\alpha,\beta)$ the equation $\Theta(\chi,\alpha,\beta)=(1,\zeta^\Gamma)$, for any possible sign of $\zeta^\Gamma$. Recall that the complex moments are purely imaginary.\\
	
	Given a parametrized tropical curve $h_0:\Gamma_0\rightarrow N_\RR$ of degree $\Delta(s)$ with $\text{ev}(\Gamma_0)=\mu$, and a real parametrized curve $h:\Gamma\rightarrow N_\RR$ having the same image $C_\text{trop}=h(\Gamma)=h(\Gamma_0)$, we say that $(\chi_0,\alpha_0,\beta_0)$ is a first order solution if $\Theta(\chi_0,\alpha_0,\beta_0)=(1,\zeta^\Gamma)$ mod $t$. Notice that, as $\mu$ is generic, the image tropical curve $C_\text{trop}$ is a nodal curve. The real rational curves with image $C_\text{trop}$ are described in Lemma \ref{realstruc}. Because of the assumption that all non-real points are on the same divisor, the graph $\Gamma_\text{even}$ consists only in the even unbounded ends. Thus, the only possibility for $\Gamma$ is to be obtained from $\Gamma_0$ by splitting the even unbounded ends, leading either to a trivalent flat vertex somewhere on the unbounded end, or to a quadrivalent one if the splitting occurs at the unique vertex adjacent to the unbounded end. Assume that $(\Gamma,h)$ has no flat vertex. We prove in the next section that it is a necessary condition for $(\Gamma,h)$ to have first order solutions. We then have the following theorem, that allows us to lift first order solutions to true solutions, provided that the Jacobian is invertible.
	
	\begin{theo}
	\label{realization theorem}
	For each real parametrized tropical curve $(\Gamma,h)$ of degree $\Delta$, obtained from a parametrized tropical curve of degree $\Delta(s)$ passing through $\mu$, without any flat vertex, and each first order solution $(\chi_0,\alpha_0,\beta_0)$, the Jacobian of $\Theta$ at $(\chi_0,\alpha_0,\beta_0)$ is invertible at first order, and there is a unique lift of $(\chi_0,\alpha_0,\beta_0)$ to a true solution $(\chi,\alpha,\beta)$ in $\mathcal{X}\times\mathcal{A}\times\mathcal{B}$.
	\end{theo}
	
	\begin{rem}
	Extending the theorem to the case of non-real points on several toric divisors would require more computations for the following reason: the Jacobian matrix may not be invertible at first order anymore. Thus, in order to find lifts of first order solutions, one would need to find solutions at a higher order,\textit{e.g.} mod $t^2$ or $t^3$.
	\end{rem}
	
	\begin{proof}
	Let $h:\Gamma\rightarrow N_\RR$ be one of the above real parametrized tropical curves, meaning that $h(\Gamma)$ is a plane tropical curve such that its parametrization $h_0:\Gamma_0\rightarrow N_\RR$ as a curve of degree $\Delta(s)$ satisfies $\text{ev}(\Gamma_0)=\mu$, and that $(\Gamma,h)$ has no flat vertex. Let $(\chi_0,\alpha_0,\beta_0)$ be a first order solution. Using Hensel's lemma, one needs to check that the Jacobian of the map is invertible at $(\chi_0,\alpha_0,\beta_0)$ in order to lift it to a unique true solution. We thus show by induction that the Jacobian matrix is invertible. For simplicity, and because of the multiplicative nature of the map $\Theta$, we use logarithmic coordinates for every variable except $\beta$. It means that we look at $\log\Theta$, depending on the new variables $(\log\chi,\log\alpha,\beta)$. Each time, the logarithm is taken coordinate by coordinate.\\
	
	Notice that if $\gamma$ is an edge, $\log\alpha_\gamma$ and $\beta_\gamma$ are scalars, while if $w$ is a real vertex, $\log\chi_w:M\rightarrow\RR$ is an element of $N_\RR$.\\
	
To compute the Jacobian	relative to coordinates $(\log\chi,\log\alpha,\beta)$ at $t=0$, we can first put $t=0$. Thus, we get for every bounded edge $\gamma\in\Gamma^1_b$:
$$\phi_\gamma|_{t=0}=\prod_{ \substack{ \gamma'\neq\gamma \\ \mathfrak{t}(\gamma')=\tg }} (\beta_{\gamma'}|_{t=0}-\beta_\gamma|_{t=0})^{n_{\gamma'}}\in N\otimes\RR^*,$$
and for every unbounded end $e_j$:
$$\varphi_j|_{t=0}=\pm \prod_{ \substack{ \gamma'\neq e_j \\ \mathfrak{t}(\gamma')=v_j }} (\beta_{e_j}|_{t=0}-\beta_{\gamma'}|_{t=0})^{\omega(n_j,n_{\gamma'})}\in\CC^*.$$
Notice that at the first order, $\phi_\gamma$ depends only on $\beta$ and not $\alpha$. For convenience of notation, all the following computations are taken at the first order, and we drop $"|_{t=0}"$ out of the notation.\\
	
	The description of $\Gamma$ implies that it has only real vertices and real bounded edges. For a bounded edge $\gamma\in\Gamma^1_b$, let $N_\gamma=\phi_\gamma\cdot\frac{\chi_\tg}{\chi_\hg}\cdot\alpha_\gamma^{n_\gamma}\in N_{\RR[[t]]^\times}$. Similarly, let $X_j=\chi_{v_j}(m_j)\varphi_j\in\RR[[t]]^\times$ for a real end $x_j$, and $Z_j=\chi_{v_j}(m_j)\varphi_j\in\CC[[t]]^\times$ for a complex end $z_j^\pm$. The variables $N_\gamma,X_j$ and $Z_j$ index the lines of the Jacobian matrix $\parfrac{\log\Theta}{(\log\chi,\log\alpha,\beta)}$. We now compute these lines.
	
	\begin{itemize}[label=-]
	\item For a general bounded edge $\gamma\in\Gamma^1_b$, one has at the first order
	$$\log N_\gamma=\sum_{ \substack{ \gamma'\neq\gamma \\ \mathfrak{t}(\gamma')=\tg }}  n_{\gamma'}\log(\beta_{\gamma'}-\beta_\gamma) \ +\ \log\chi_\tg-\log\chi_\hg+n_\gamma\log\alpha_\gamma \in\text{Hom}(M,\RR).$$
	Therefore, one has the following partial Jacobian matrices,
	$$\parfrac{\log N_\gamma}{\log\chi_\tg}=I_2 ,\ \parfrac{\log N_\gamma}{\log\chi_\hg}=-I_2 , \ \parfrac{\log N_\gamma}{\log\alpha_\gamma}=n_\gamma.$$
	Concerning the $\beta$ variables, the only cases when $N_\gamma$ depends on a $\beta_{\gamma'}$ parameter at first order is when the tail $\tg$ of $\gamma$ is one of the quadrivalent vertices. In that case, the intervening coordinate is precisely $\beta_\gamma$. Then one has
	$$\parfrac{\log N_\gamma}{\beta_\gamma}=\frac{2\beta_\gamma}{\beta_\gamma^2+1}n_{\gamma'},$$
	where $n_{\gamma'}\in N$ is the vector directing the complex ends adjacents to $\tg$.
	
	\item For a real end $x_j$, corresponding to an end $e_j$, directed by $n_j$ and adjacent to a vertex $v_j$, one has at the first order
	$$\log X_j=\sum_{ \substack{ \gamma'\neq e_j \\ \mathfrak{t}(\gamma')=v_j }} \omega(n_j,n_{\gamma'}) \log(\beta_{e_j}-\beta_{\gamma'}) \ +\ \log\chi_{v_j}(m_j)\in\RR.$$
	The Jacobian relative to $\log\chi\in N\otimes\RR$ is the Jacobian of the evaluation at $m_j$, which is linear with respect to $\log\chi$. Thus, one has
	$$\parfrac{\log X_j}{\log\chi_{v_j}}=m_j=\iota_{n_j}\omega.$$
	Now, the only case where $\log X_j$ depends on a $\beta$ parameter is when $v_j$ is a quadrivalent vertex. In that case, the parameter is $\beta_{e_j}$. Let $n_{\gamma'}$ be the vector directing the complex ends adjacent to $v_j$. Then, one has
	$$\parfrac{\log X_j}{\beta_{e_j}}=\omega(n_j,n_{\gamma'})\frac{2\beta_{e_j}}{\beta_{e_j}^2+1}.$$
	
	\item For the case of $Z_j$, the computations are similar, only this time the target space is $\CC$ instead of $\RR$. The unbounded end $e_j$ is adjacent to a quadrivalent vertex. Let $\gamma$ be the real edge such that $\tg=v_j$. This edge might be unbounded. Then, one has
	$$\log Z_j=\sum_{ \substack{ \gamma'\neq e_j \\ \mathfrak{t}(\gamma')=v_j }} \omega(n_j,n_{\gamma'}) \log(\beta_{e_j}-\beta_{\gamma'}) \ +\ \log\chi_{v_j}(m_j)\in\CC.$$
	Thus, we once again have
	$$\parfrac{\log Z_j}{\log\chi_{v_j}}=m_j=\iota_{n_j}\omega\in M_\RR\subset M_\CC,$$
	and this time
	$$\parfrac{\log Z_j}{\beta_\gamma}=\omega(n_\gamma,n_j)\parfrac{\log(i-\beta_\gamma)}{\beta_\gamma}= \frac{\omega(n_\gamma,n_j)}{\beta_\gamma^2+1}\begin{pmatrix}
\beta_\gamma \\
1 \\	
\end{pmatrix}=	\frac{\omega(n_\gamma,n_j)}{\beta_\gamma^2+1}(\beta_\gamma+i)\in\CC .$$
	\end{itemize}
Now that all the terms of the Jacobian matrix $\parfrac{\log\Theta}{(\log\chi,\log\alpha,\beta)}$ are known, we make an induction on the number of vertices to prove that it is invertible. The initialization is done within the induction step, by removing the column indexed by $\alpha_\gamma$, and the remaining rows and columns which are not drawn on the array. Let $V$ be a vertex of $\Gamma$ which is adjacent to two real unbounded ends, or one real and two complex unbounded ends, as depicted on Figure \ref{induction vertex}.

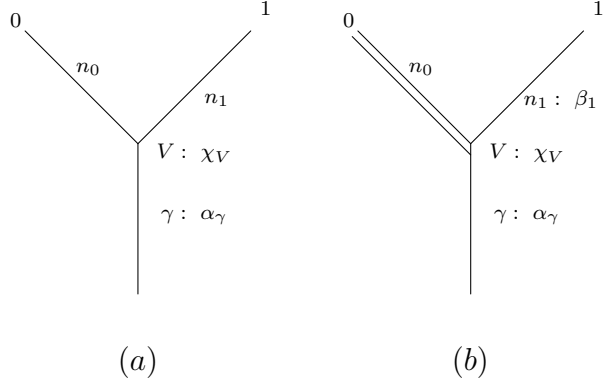
\begin{figure}
\begin{center}
\begin{tabular}{cc}
\begin{tikzpicture}[line cap=round,line join=round,>=triangle 45,x=0.5cm,y=0.5cm]
\clip(-1.,-1.) rectangle (7.,8.);
\draw (3.,0.)-- (3.,4.);
\draw (3.,4.)-- (0.,7.);
\draw (3.,4.)-- (6.,7.);
\begin{scriptsize}
\draw[color=black] (4.5,3.744529205178832) node {$V:\ \chi_V$};
\draw[color=black] (4.5,2.0531972812887695) node {$\gamma:\ \alpha_\gamma$};
\draw[color=black] (-0.24508937151643118,7.280950500585326) node {$0$};
\draw[color=black] (1.6838676987052943,5.99497912043751) node {$n_0$};
\draw[color=black] (6.4,7.6) node {$1$};
\draw[color=black] (5.1,5.170280083168802) node {$n_1$};
\end{scriptsize}
\end{tikzpicture}
&
\begin{tikzpicture}[line cap=round,line join=round,>=triangle 45,x=0.5cm,y=0.5cm]
\clip(-1.,-1.) rectangle (7.,8.);
\draw (3.,0.)-- (3.,4.);
\draw (3.,4.)-- (0.,7.);
\draw (3.,3.7)-- (-0.15,6.85);
\draw (3.,4.)-- (6.,7.);
\begin{scriptsize}
\draw[color=black] (4.5,3.744529205178832) node {$V:\ \chi_V$};
\draw[color=black] (4.5,2.0531972812887695) node {$\gamma:\ \alpha_\gamma$};
\draw[color=black] (-0.24508937151643118,7.280950500585326) node {$0$};
\draw[color=black] (1.6838676987052943,5.99497912043751) node {$n_0$};
\draw[color=black] (6.4,7.6) node {$1$};
\draw[color=black] (5.4,5.170280083168802) node {$n_1:\ \beta_1$};
\end{scriptsize}
\end{tikzpicture}
\\
$(a)$ & $(b)$ \\
\end{tabular}
\label{induction vertex}
\caption{Vertices adjacent to two real ends $(a)$ or a real end and two complex ends $(b)$.}
\end{center}
\end{figure}

\begin{itemize}[label=-]
\item Let $\gamma$ be the edge with $\hg=V$. We assume that $\gamma$ is a bounded edge. In the first case, the matrix has the following form

$$\text{Jac}\Theta=\begin{array}{lcccccc}
                                        & \multicolumn{1}{l}{} & \multicolumn{1}{l}{} & \multicolumn{1}{l}{}   & \multicolumn{2}{c}{\chi_V}            & \multicolumn{1}{c}{\alpha_\gamma}                 \\ \cline{5-7} 
                                        &                      &                      & \multicolumn{1}{c|}{}  & 0     & \multicolumn{1}{c|}{0}     & \multicolumn{1}{c|}{0}                    \\
                                        &                      & \ast                  & \multicolumn{1}{c|}{}  & \vdots & \multicolumn{1}{c|}{\vdots} & \multicolumn{1}{c|}{\vdots}                \\
                                        &                      &                      & \multicolumn{1}{c|}{}  & 0     & \multicolumn{1}{c|}{0}     & \multicolumn{1}{c|}{0}                    \\ \cline{2-7} 
\multicolumn{1}{r|}{\multirow{2}{*}{$N_\gamma$}} &        \ast              &                      \cdots & \multicolumn{1}{c|}{\ast}  & -1    & \multicolumn{1}{c|}{0}     & \multicolumn{1}{c|}{\multirow{2}{*}{$n_0+n_1$}} \\
\multicolumn{1}{r|}{}                   &        \ast              &                     \cdots & \multicolumn{1}{c|}{\ast}  & 0     & \multicolumn{1}{c|}{-1}    & \multicolumn{1}{c|}{}                     \\ \cline{2-7} 
\multicolumn{1}{r|}{C_0}                  & 0                    & \cdots                & \multicolumn{1}{c|}{0} & \multicolumn{2}{c|}{m_0}             & \multicolumn{1}{c|}{0}                    \\ \cline{2-7} 
\multicolumn{1}{r|}{C_1}                  & 0                    & \cdots                & \multicolumn{1}{c|}{0} & \multicolumn{2}{c|}{m_1}             & \multicolumn{1}{c|}{0}                    \\ \cline{2-7} 
\end{array}.$$
By developing with respect to the last two rows, since $(m_0,m_1)$ are free, we are left with the following determinant:
$$\begin{array}{lcccc}
                                        & \multicolumn{1}{l}{} & \multicolumn{1}{l}{} & \multicolumn{1}{l}{}     & \multicolumn{1}{l}{\alpha_\gamma}                 \\ \cline{5-5} 
                                        &                      &                      & \multicolumn{1}{c|}{}    & \multicolumn{1}{c|}{0}                    \\
                                        &                      & \ast                  & \multicolumn{1}{c|}{}    & \multicolumn{1}{c|}{\vdots}                \\
                                        &                      &                      & \multicolumn{1}{c|}{}    & \multicolumn{1}{c|}{0}                    \\ \cline{2-5} 
\multicolumn{1}{r|}{\multirow{2}{*}{$N_\gamma$}} & \ast                  & \cdots                & \multicolumn{1}{c|}{\ast} & \multicolumn{1}{c|}{\multirow{2}{*}{$n_\gamma$}} \\
\multicolumn{1}{r|}{}                   & \ast                  & \cdots                & \multicolumn{1}{c|}{\ast} & \multicolumn{1}{c|}{}                     \\ \cline{2-5} 
\end{array}.$$
If $\gamma$ was an unbounded end, we would be left with the empty matrix and we would have proven invertibleness. Otherwise, the last two rows correspond to a copy of $N_\RR$, and are thus given by two elements of $M_\RR$, the dual of $N_\RR$. Up to a change of basis, one can assume that one of these elements of $M_\RR$ is $\omega(n_\gamma,-)$, which takes $0$ value on $n_\gamma$. Thus, by making a development with respect to the column, we are reduced to the determinant matrix where the bounded edge $\gamma$ is replaced with a unbounded real end, directed by $n_\gamma$, and $N_\gamma$ is replaced by the evaluation of the moment of this new unbounded end. Thus, the matrix is invertible by induction.

\item If the vertex $V$ is adjacent to two complex unbounded ends (directed by $n_0$), and to a real unbounded end (directed by $n_1$), let $\beta_1$ be the $\beta$ coordinate associated to the real end. Let $\gamma$ be the edge with $\hg=V$, thus directed by $2n_0+n_1$. The determinant takes the following form:
$$\text{Jac}\Theta=\begin{array}{lccccccc}
                                        &     &       &                          & \multicolumn{2}{c}{\chi_V}            & \alpha_\gamma                                   & \beta_1                       \\ \cline{5-8} 
                                        &     &       & \multicolumn{1}{c|}{}    & 0     & \multicolumn{1}{c|}{0}     & \multicolumn{1}{c|}{0}                  & \multicolumn{1}{c|}{0}     \\
                                        &     & \ast   & \multicolumn{1}{c|}{}    & \vdots & \multicolumn{1}{c|}{\vdots} & \multicolumn{1}{c|}{\vdots}              & \multicolumn{1}{c|}{\vdots} \\
                                        &     &       & \multicolumn{1}{c|}{}    & 0     & \multicolumn{1}{c|}{0}     & \multicolumn{1}{c|}{0}                  & \multicolumn{1}{c|}{0}     \\ \cline{2-8} 
\multicolumn{1}{l|}{\multirow{2}{*}{$N_\gamma$}} & \ast & \cdots & \multicolumn{1}{c|}{\ast} & -1    & \multicolumn{1}{c|}{0}     & \multicolumn{1}{c|}{\multirow{2}{*}{$n_\gamma$}} & \multicolumn{1}{c|}{0}     \\
\multicolumn{1}{l|}{}                   & \ast & \cdots & \multicolumn{1}{c|}{\ast} & 0     & \multicolumn{1}{c|}{-1}    & \multicolumn{1}{c|}{}                   & \multicolumn{1}{c|}{0}     \\ \cline{2-8} 
\multicolumn{1}{l|}{\multirow{2}{*}{$C_0$}} & 0   & \cdots & \multicolumn{1}{c|}{0}   & \multicolumn{2}{c|}{m_0}             & \multicolumn{1}{c|}{0}                  & \multicolumn{1}{c|}{\beta_1}  \\
\multicolumn{1}{l|}{}                   & 0   & \cdots & \multicolumn{1}{c|}{0}   & 0     & \multicolumn{1}{c|}{0}     & \multicolumn{1}{c|}{0}                  & \multicolumn{1}{c|}{1}     \\ \cline{2-8} 
\multicolumn{1}{l|}{C_1}                  & 0   & \cdots & \multicolumn{1}{c|}{0}   & \multicolumn{2}{c|}{m_1}             & \multicolumn{1}{c|}{0}                  & \multicolumn{1}{c|}{2\beta_1}  \\ \cline{2-8} 
\end{array}.
$$
Notice that we dropped out the constant factor $\frac{\omega(2n_0,n_1)}{\beta_1^2+1}$ in the last column. Similarly to the previous case, one can make a development with respect to the penultimate row, and then the second resulting last rows. We recover the same determinant as in the case of a trivalent real vertex, which is also the empty determinant if $\gamma$ is an unbounded end. Hence, the conclusion follows, reducing once again to a graph with one vertex less.
\end{itemize}
	\end{proof}

	\section{Statement of the result and proof}

	\subsection{Statement of result and main proof}
	
	We now can prove Theorem \ref{theorem paper}, relating $R_{\Delta,s}$ and $N_{\Delta(s)}^{\partial,\text{trop}}$.\\

Let $\mathcal{P}_t$ be a symmetric configuration of points depending on a parameter $t$, chosen as in section \ref{classical problem}. This means that one is given a collection of series $\pm \zeta_i(t)\in\RR((t))^*\cup i\RR((t))^*$ corresponding to the coordinates of the points $\pm p_i(t)$ of $\mathcal{P}_t$ on the toric divisors. Let $\mu$ be the tropicalization of the point configuration, \textit{i.e.} for a pair of points $\pm p_i(t)$, we have $\mu_i=\text{val} \zeta_i(t)$. The correspondence theorem, proven in the previous section, provides for $t$ large enough a correspondence between the curves of $\mathcal{S}(\mathcal{P}_t)$, which are real parametrized curves of degree $\Delta$, and the parametrized tropical curves $\Gamma$ of degree $\Delta(s)$ such that $\text{ev}(\Gamma)=\mu$. This is done by allowing one to lift every first order solution to a true solution. We here assign a multiplicity to each curve of $\text{ev}^{-1}(\mu)$, so that the count of $\text{ev}^{-1}(\mu)$ with these multiplicities gives the invariant $R_{\Delta,s}$. This multiplicity happens to be proportional to the refined multiplicity of Block-G\"ottsche, thus leading to the relation stated in Theorem \ref{theorem paper}.\\

Being given a parametrized tropical curve $h_0:\Gamma_0\rightarrow N_\RR$ of degree $\Delta(s)$ with $\text{ev}(\Gamma_0)=\mu$, the task of computing its multiplicity amounts to two things. The first is to find the parametrized real tropical curves of degree $\Delta$ having the same image. The second task consists in finding the first order solutions to $\Theta(\chi,\alpha,\beta)=(1,\zeta^\Gamma)$.\\

Let $h_0:\Gamma_0\rightarrow N_\RR$ be a parametrized rational tropical curve of degree $\Delta(s)$ such that $\text{ev}(\Gamma_0)=\mu$. Let $C_\text{trop}=h_0(\Gamma_0)$ be its image, which is a plane tropical curve. We need to find the real parametrized tropical curves $(\Gamma,h)$ of degree $\Delta$ parametrizing the plane curve $C_\text{trop}$ and admitting a first order solution. The different possible real structures are described in Proposition \ref{realstruc}. 

\begin{lem}
The parametrized real tropical curves with a trivalent flat vertex cannot be the tropicalization of a family of parametrized real rational curves passing through $\mathcal{P}_t$.
\end{lem}

\begin{proof}
Assume that there exists a trivalent flat vertex $w$, in direction $n_{j_0}$, with two outgoing unbounded ends exchanged by the involution. Let $f_t:\CC P^1\dashrightarrow N\otimes\CC((t))^*$ be a parametrized real rational curve tropicalizing to $(\Gamma_0,h_0)$. Then, in the real coordinate $y$ such that the two conjugated points have coordinate $\pm i$ and some real point has coordinate $\infty$, the morphism takes the following form:
$$f(y)=\chi_wt^{h(w)}(y^2+1)^{n_{j_0}}\prod_{j}\left( \frac{y}{y(q_j)}-1\right)^{n_j}\in\text{Hom}(M,\CC((t))^*),$$
where $q_j$ are the other boundary points of the curve, and $\chi_w\in N\otimes\RR((t))^*$. The moment at $\pm i$ is obtained by evaluating at $\iota_{n_{j_0}}\omega\in M$, then at $\pm i$. At the first order, the moments of the exchanged unbounded ends are real: the big product takes value $1$, the coefficient $\chi_w(\iota_{n_{j_0}}\omega)$ is real, and $(y^2+1)^{n_{j_0}}$ is evaluated with $0$ exponent. This is absurd since it is supposed to be purely imaginary. Hence, we cannot have any flat vertex.
\end{proof}

\begin{lem}
\label{description real solution}
Among the real parametrized tropical curves $h:\Gamma\rightarrow N_\RR$ of degree $\Delta$ with image $C_\text{trop}$, at most one may be the tropicalization of a family of parametrized real rational curves passing through $\mathcal{P}_t$. Moreover, this real tropical curve is the tropical curve obtained from $\Gamma_0$ with the maximal splitting graph.
\end{lem}

\begin{proof}
There is an infinite number of parametrized curves with image $C_\text{trop}$, obtained by splitting the graph of even edges and described in Proposition \ref{realstruc}. All the unbounded ends of $C_\text{trop}$ associated to the complex markings are double edges near infinity since they correspond to two distinct marked points. Thus they belong to $\Gamma_\text{even}$. Since all the complex markings are on the same divisor, they have the same direction and they cannot meet at a common vertex. Therefore, there are no extendable vertex  and the graph $\Gamma_\text{even}$ only consists of the even unbounded ends. The only possibility is that the double ends separates itself at a trivalent flat vertex, sent somewhere on the unbounded end of $C_\text{trop}$. However, this is forbidden by the previous lemma. Therefore, all the double ends split and there is a unique possibility.
\end{proof}

We have proven that for $C_\text{trop}$, there is a unique real parametrized tropical curve $h:\Gamma\rightarrow N_\RR$ of degree $\Delta$ with image $C_\text{trop}$ that can be the tropicalization of a family of parametrized real rational curves. We now count the first order solutions, which lead to true solutions thanks to the correspondence theorem. The count of these solutions gives the multiplicity used to count tropical curves.

\begin{prop}\label{firstordersol}
Let $\mathcal{P}_t$ be a symmetric real configuration of points as previously chosen, tropicalizing on a family of moments $\mu$. Let $h_0:\Gamma_0\rightarrow N_\RR$ be a parametrized tropical curve of degree $\Delta(s)$ having moments $\mu$, and let $h:\Gamma\rightarrow N_\RR$ be the associated real parametrized tropical curve without flat vertex such that $\text{ev}(\Gamma)=\mu$. Vertices of $\Gamma$ and $\Gamma_0$ are canonically identified. Let $W_1,\dots,W_{s}$ be the quadrivalent vertices of $\Gamma$, adjacent to the complex unbounded ends, let $m_{W_i}$ denotes their complex multiplicity as a trivalent vertex of $\Gamma_0$.
Then there are precisely $2^{m-2s}\prod m_{W_i}$ oriented real curves passing through the symmetric configuration $\mathcal{P}_t$ and tropicalizing to $(\Gamma,h)$. Their refined count according to the quantum index and sign $\sigma(\overrightarrow{C})$ is given by
$$m'_\Gamma=4\prod_1^{s}\frac{\q{m_{W_i}}}{q-q^{-1}}\prod_{V\neq W_i}(\q{m_V}).$$
\end{prop}

Before proving this proposition, we can now prove Theorem \ref{theorem paper}.

\begin{proof}[Proof of Theorem \ref{theorem paper}]
This is a consequence of the correspondence theorem that states that each of the first order solutions lifts to a unique solution, given by Proposition \ref{firstordersol}, and an easy computation between the multiplicities: we obtain $R_{\Delta,s}$ by counting curves with multiplicities $\frac{1}{4}m'_\Gamma$. The multiplicity $\frac{1}{4}m'_\Gamma$ is obtained from $m_\Gamma^q$ by clearing the denominators of the $m-2-s$ vertices and dividing by the terms of the $s$ quadrivalent vertices:
$$\frac{1}{4}m'_\Gamma=\frac{(\qd)^{m-2-s}}{(q-q^{-1})^{s}} m^q_\Gamma.$$
Therefore, one has
$$R_{\Delta,s}=\frac{(\qd)^{m-2-s}}{(q-q^{-1})^{s}}N^{\partial,\text{trop}}_{\Delta(s)}.$$
Using the identity $q-q^{-1}=(\qd)(q^\frac{1}{2}+q^{-\frac{1}{2}})$, we get that
$$R_{\Delta,s} =\frac{(\qd)^{m-2-2s}}{(q^\frac{1}{2}+q^{-\frac{1}{2}})^{s}}N^{\partial,\text{trop}}_{\Delta(s)}.  $$
\end{proof} 

In \cite{gottsche2019refined} L. G\"ottsche and F. Schroeter proposed a refined way to count so-called \textit{refined Broccoli curves} having fixed ends, and passing through a fixed configuration of "real and complex" points. In the case where there are only marked ends and no marked points, this count coincides with the count of plane tropical curves passing through the configuration with usual Block-G\"ottsche multiplicities from \cite{gottsche2014refined} up to a multiplication by a constant term depending on the degree and easily computed. More precisely, provided there are no marked points, the refined Broccoli multiplicity is just the refined multiplicity from \cite{gottsche2014refined} enhanced by a product over the ends of weight higher than $2$, coinciding with this aformentionned constant term. In our case, since the only multiple edges are marked and of weight $2$ (only one real unbounded end is unmarked and of weight one), this factor is $\frac{q+q^{-1}}{q^{1/2}+q^{-1/2}}$ for each of the $s$ ends. If we denote by $BG_{\Delta(s)}(q)$ the refined invariant obtained in \cite{gottsche2019refined}, then we have the relation
$$R_{\Delta,s}(q) = \frac{(\qd)^{m-2-2s}}{(q+q^{-1})^{s}} BG_{\Delta(s)}(q).$$

We now prove Proposition \ref{firstordersol}.

\begin{proof}[Proof of Proposition \ref{firstordersol}]
The proof is made with an induction on the number of vertices of the curve $\Gamma$. Exceptionally, to suit the induction, a vector $n_j$ of $N$ directing an unbounded end $e_j$ may not be primitive. In that case, we denote by $m_j=\omega\left(\frac{n_j}{l(n_j)},-\right)$ the dual vector, but still of lattice length $1$. Let $V$ be a vertex adjacent to two real ends, or one real end and two complex ends. Let $\gamma$ be the edge, maybe unbounded, with $\hg=V$.

\begin{itemize}[label=-]
\item If $V$ is a real vertex adjacent to two real unbounded ends indexed by $0$ and $1$, then we have to solve for $\chi_V:M\rightarrow\RR^*$ the following system:
$$\left\{\begin{array}{l}
\chi_V(m_0)=\pm\zeta_0^\Gamma\in\RR^* \\
\chi_V(m_1)=\pm\zeta_1^\Gamma\in\RR^* \\
\end{array}\right. .$$
Recall that the vectors $n_0$ and $n_1$ might not be primitive, but $m_0$ and $m_1$ are. This system leads to $4$ solutions: if $(e_1^*,e_2^*)$ is a basis of $M$, the absolute value of $\chi_V(e_1^*)$ and $\chi_V(e_2^*)$ is uniquely determined, while the sign may be chosen arbitrarily. Notice that some choices of signs for the right-hand side of the system may provide several solutions, while other provide none. Let $m_\gamma$ be the primitive vector dual to $n_\gamma$. Let $\tilde{m}_\gamma$ be such that $(m_\gamma,\tilde{m}_\gamma)$ is a basis of $M$. These $4$ solutions separate themselves into two groups of $2$, according to the sign of $\chi_V(m_\gamma)$. If $\gamma$ is unbounded, this closes the proof.\\

Now, if $\gamma$ is bounded, we have the equation
$$\phi_\gamma\cdot\frac{\chi_\tg}{\chi_V}\cdot\alpha_\gamma^{n_\gamma}=1\in N\otimes\RR^*.$$
We evaluate at $m_\gamma$, leading to
$$\phi_\gamma(m_\gamma)\chi_\tg(m_\gamma)=\chi_V(m_\gamma)\in\RR^*.$$
Recall that according to the choice of signs of $\pm\zeta_j$, the sign of $\chi_V(m_\gamma)$ may change. Let replace the bounded edge $\gamma$ by an unbounded end with direction $n_\gamma$, leading to a new parametrized tropical curve $(\Gamma',h')$. The above  equation is the equation associated to this new unbounded end in $\Gamma'$, in the corresponding system $\Theta(\chi,\alpha,\beta)=(1,\zeta^\Gamma)$. Thus, we can proceed by induction. Let $4R$ denote the refined signed count of oriented curves lifting $\Gamma'$. These $4R$ curves separate themselves into four groups of $R$ according to the value of the signs the function $\phi_\gamma\cdot\chi_\tg$ takes on the basis $(m_\gamma,\tilde{m}_\gamma)$.\\

Last, we need to solve for $\alpha_\gamma$. By evaluating at $\tilde{m}_\gamma$, we get
$$\alpha_\gamma^{\langle n_\gamma,\tilde{m}_\gamma\rangle}=\frac{\chi_V(\tilde{m}_\gamma)}{\phi_\gamma(\tilde{m}_\gamma)\cdot\chi_\tg(\tilde{m}_\gamma)}.$$
The solving as well as the number of solutions depends on the sign of the right-hand side.
	\begin{itemize}[label=$\star$]
	\item If $n_\gamma$ has odd integer length, then we solve uniquely for $\alpha_\gamma$ for each possible sign of $\chi_V(\tilde{m}_\gamma)$. The sign of $\chi_V(m_\gamma)$ is already determined since we have $\phi_\gamma(m_\gamma)\chi_\tg(m_\gamma)=\chi_V(m_\gamma)$. Thus, each of the oriented curves in each of the groups have two possible solutions for $\chi_V$. This corresponds to the gluing of two possible curves, one increasing the logarithmic rotation number by one, the other decreasing it by one. In each case, the orientation of the curve propagates, and the signed count becomes
	$$4\times (q^{m_V/2}-q^{-m_V/2})R.$$
	
	\item If $n_\gamma$ has even integer length, then the sign of $\chi_V(m_\gamma)$ is still determined, and the sign of $\chi_V(\tilde{m}_\gamma)$ is forced in order to have at least one solution for $\alpha_\gamma$. In that case, we have two. The two possible choices of $\alpha_\gamma$ correspond to the two ways of gluing a curve over $V$ to one of the $4R$ curves, when such a gluing is possible. One of these choices decreases the logarithmic rotation number by one while the other increases it by one. The signed count then becomes again
	$$4\times (q^{m_V/2}-q^{-m_V/2})R.$$
	\end{itemize}

\item If $V$ is adjacent to two complex ends and a real end, respectively indexed by $0$ and $1$, we then have to solve for $\beta_1$ and $\chi_V$ the system
$$\left\{\begin{array}{l}
\chi_V(m_1)(\beta_1^2+1)^{\langle n_0,m_1\rangle}=\zeta_0^\Gamma\in i\RR^* \\
\chi_V(m_0)(i-\beta_1)^{\langle n_1,m_0\rangle}=\pm\zeta_1^\Gamma\in\RR^* \\
\end{array}\right. .$$
This system was already solved in section \ref{local computations}. Assume that the degree $\{n_0,2n_1,-n_0-2n_1\}$ is equivalent to $\Delta(m_0,m_1,m_2)$ as in section \ref{local computations}, ignoring the temporary conflict of notation $m_i$. (They are integers in $\Delta(m_0,m_1,m_2)$ and co-characters in $\chi_V(m_i)$) The logarithmic rotation number of the solutions is equal to $0$, so each solution is counted with a positive sign. The refined count of the solutions from section \ref{local computations} is equal to $0$, $1$ or $2$ times the following sum, which covers all the possible values of the quantum index:
$$\sum_{k=0}^{m_1-1}q^{(m_3-m_2)(2k+1-m_1)}=\frac{q^{(m_3-m_2)m_1-q^{-(m_3-m_2)m_1}}}{q^{m_3-m_2}-q^{-({m_3-m_2})}}=\frac{\q{m_W}}{q-q^{-1}},$$
	since $m_3-m_2=1$, and the complex multiplicity $m_W$ satisfies $m_W=2m_1$, which is the multiplicity of the vertex. Accounting for both possible orientations, and both choices of signs for the $\zeta_i$, this closes the proof if $\gamma$ is an unbounded end. Otherwise, we then use the equation
$$\phi_\gamma\cdot\frac{\chi_\tg}{\chi_V}\cdot\alpha_\gamma^{n_\gamma}=1\in N\otimes\RR^*,$$
just as in the previous step. We solve for the other unknowns inductively, then for $\alpha_\gamma$. The same disjunction provides the new signed count
$$4\frac{q^{m_V/2}-q^{-m_V/2}}{q-q^{-1}}R.$$
\end{itemize}
\end{proof}

\subsection{Alternative proof}

So far, we have proven a correspondence theorem using the approach of Tyomkin in \cite{tyomkin2012tropical}. One could also carry a proof of Theorem \ref{theorem paper} using the approach of Mikhalkin \cite{mikhalkin2005enumerative,mikhalkin2017quantum}, or Shustin \cite{shustin2004tropical}. Both adopt a description of plane curves by polynomial equations rather than a parametrization.\\ 

Briefly, given a real parametrized tropical curve, this method consists in finding a collection of plane curves indexed by the vertices of the tropical curves, and a way of gluing them along the edges of the curve. We refer the reader to \cite{mikhalkin2005enumerative,mikhalkin2017quantum,shustin2004tropical} for details and proofs. Here, we recover the multiplicity $m'_\Gamma$ by counting the possible families of curves indexed by the vertices, and the number of gluings. Using the aforementioned approach to the correspondence theorem, this computation reproves Theorem \ref{theorem paper} through a new proof of Proposition \ref{firstordersol}.\\

\begin{proof}[Alternative Proof of Proposition \ref{firstordersol}] 
We make an induction on the number of vertices. Thus, we initialize with curves $\Gamma$ having a unique vertex, trivalent or quadrivalent.\\

Following \cite{mikhalkin2005enumerative}, to compute the multiplicity, one needs to count (in a suitable way) the local curves over the vertices of $\Gamma$, and the number of ways to glue them together. In this proof, we do not assume the vectors of $\Delta$ to be primitive.

	\begin{itemize}[label=-]
	\item If there is only one trivalent vertex $V$ in $\Gamma$, then we are looking for curves maximally tangent to the toric divisors and passing through two pairs of opposite real points. There are $4$ such curves, which are exchanged by the action of the deck transformation group $\{\pm 1\}^2$ on the associated toric surface. Notice that if we specify the points in the pairs the number of curves may vary, but if we consider both points in the pair, the number remains the same. These $4$ curves lead to $8$ oriented curves. Half of them have logarithmic Gauss degree $1$ and the other half has degree $-1$, leading to signs $\sigma(\RR C)=1$ or $-1$. Therefore the signed contribution is $4(\q{m_V})$.
	\item If there is only one quadrivalent vertex $W$, we look for curves passing through one pair of conjugated imaginary points on one divisor, and a pair of opposite real points. Assume that the degree of the vertex is $\Delta(m_1,m_2,m_3)$. Then, as the unbounded non-real ends are of weight $2$, we have $m_3-m_2=1$, and the complex multiplicity is $m_W=2m_1$. We have seen that there are always $2m_1$ curves passing through the configuration: either $m_1$ for each of the real points in the pair, or $2m_1$ and $0$. Therefore, there are $4m_1$ oriented curves going through the pair. Moreover, their quantum index is known. The logarithmic rotation number $\text{Rot}_\text{Log}$ can be computed thanks to the same monomial map that allowed us to compute their quantum index: if $A$ denotes the matrix of the monomial map, then
	$$\text{Rot}_\text{Log}(\psi_k) = \det A\times\text{Rot}_\text{Log}(\varphi_k).$$
	As the logarithmic rotation number of $\varphi_k$ is $0$, all the curves have logarithmic rotation number zero and therefore $\sigma(\overrightarrow{C})=1$. When accounting for both orientations, the desired count is
	$$4\sum_{k=0}^{m_1-1}q^{(m_3-m_2)(2k+1-m_1)}=4\frac{q^{(m_3-m_2)m_1-q^{-(m_3-m_2)m_1}}}{q^{m_3-m_2}-q^{-({m_3-m_2})}}=4\frac{\q{m_W}}{q-q^{-1}},$$
	since $m_3-m_2=1$, and the complex multiplicity $m_W$ satisfies $m_W=2m_1$, which is the lattice area of the dual triangle.
\end{itemize}

Now that the initialization is done, assume that $\Gamma$ has more than one vertex. Let $V$ be a vertex adjacent to two unbounded real ends, or one real end and two complex ends. The last edge adjacent to $V$, which is bounded, is denoted by $\gamma$. Let $\Gamma'$ be the parametrized tropical curve obtained by deleting this vertex and replacing the edge $\gamma$ heading to $V$ by an unbounded end with same direction. The Menelaus rule 
allows us to define a pair of real opposite moments associated to the new unbounded end. This moment is defined by the condition that the symmetric configuration composed by the pairs of points of the edges adjacent to $V$ satisfies the Menelaus condition. We get a new symmetric configuration of points $\mathcal{P}_t'$, indexed by the ends of $\Gamma'$.\\

Let $4R$ be the refined count of oriented curves tropicalizing on $\Gamma'$, passing through the symmetric configuration $\mathcal{P}_t'$. We now have to glue together the oriented curves above $\Gamma'$, and the curves over the vertex $V$. According to \cite{mikhalkin2005enumerative}, such a gluing is possible only if the \textit{phases} agree, and if the edge has weight two, there are two ways to do it. Recall that the \textit{phase} of the curve $C_V$ above $V$ with respect to the edge $\gamma$ is defined as follows: the edge $\gamma$ corresponds to an intersection point $p$ of $C_V$ with the toric boundary, the phase is the (set of) quadrants of $(\RR^*)^2$ in which the curve sits in a neighborhood of this intersection point. This means the following: the set of quadrants is identified with $\mathbb{F}_2^2$, and
\begin{itemize}[label=-]
\item If the vector $n_\gamma$, which directs $\gamma$ has odd lattice length, then the curve passes through the corresponding toric divisor and changes of quadrant at $p$. The phase is the element of $\mathbb{F}_2^2/\langle n_\gamma\rangle$ that corresponds to these quadrants.
\item If the vector $n_\gamma$ has even lattice length, the curve stays on the same side of the toric divisor, thus in the same quadrant. In that case, the phase is an element of $\mathbb{F}_2^2$.
\end{itemize}

In each case we inquire for the refined count over the global tropical curve $\Gamma$. We make a disjunction over the type of $V$, which can either be a trivalent one, or a quadrivalent one (\textit{i.e.} one of the vertices $W_i$), and $\gamma$ can be an odd or an even edge. According to \cite{mikhalkin2005enumerative}, if the edge is even, we have two ways of gluing the curves together when the phases agree. We use the action of the deck transformation group $\{\pm 1\}^2$.

\begin{itemize}[label=$\bullet$]
\item Assume that $V$ is trivalent, and the bounded edge adjacent to $V$ is odd. Then, there are two real opposite points, corresponding to the two distinct phases that the edge can have, where the gluing can happen, and they are exchanged by the deck transformation group $\{\pm 1\}^2$. Therefore, over $\Gamma'$, there are $2R$ oriented curves for each of the phases (both add up to the total $4R$ oriented curves). There are $4$ curves above the vertex $V$, two over each real phase. If a pair with compatible phases is chosen, we have a unique way of gluing. Moreover, for any possible gluing of an oriented curve and a curve, the orientation of the oriented curve extends to the new global curve. The curves above $V$ thus get an orientation. The two oriented curves obtained this way have opposite quantum indices, and one increases by one the logarithmic rotation number while the other decreases it by one. Finally, the total contribution is
$$(q^\frac{m_V}{2}-q^{-\frac{m_V}{2}})2R+(q^\frac{m_V}{2}-q^{-\frac{m_V}{2}})2R=4(q^\frac{m_V}{2}-q^{-\frac{m_V}{2}})R.$$

\item Assume that $V$ is trivalent, and $\gamma$ is an even edge. We have $4$ possible phases, exchanged by the deck transformation group. For each of these phases, there are $R$ oriented curves over $\Gamma'$, and just one curve over $V$. Each time, there are two ways of gluing the curves: one that increases the logarithmic rotation number, and one that decreases it. We thus get
$$(q^\frac{m_V}{2}-q^{-\frac{m_V}{2}})R+(q^\frac{m_V}{2}-q^{-\frac{m_V}{2}})R+(q^\frac{m_V}{2}-q^{-\frac{m_V}{2}})R+(q^\frac{m_V}{2}-q^{-\frac{m_V}{2}})R=4(q^\frac{m_V}{2}-q^{-\frac{m_V}{2}})R.$$

\item Assume that $V=W$ is a quadrivalent vertex, and that $\gamma$ is an odd edge. Assume that the dual triangle is equivalent to $\Delta(m_1,m_2,m_3)$. There are two possible phases. Over $\Gamma'$, there are $2R$ oriented curves for each of the phases, while there are $m_1$ curves over $W$ for each of the phases. In each case the gluing is unique and we get
$$\left(\frac{\q{m_W}}{q-q^{-1}}\right)2R+\left(\frac{\q{m_W}}{q-q^{-1}}\right)2R=4\left(\frac{\q{m_W}}{q-q^{-1}}\right)R.$$

\item Finally, if $V=W$ is a quadrivalent vertex and $\gamma$ is an even edge, there are four phases to consider, each one with $R$ oriented curves over $\Gamma'$. This time $m_1$ is even, and the distribution of the phases for the curves above the vertex might be a little trickier. We have seen that if the boundary points are fixed, there are either $2m_1$ curves above $W$ for one of the points and zero for the other, or $m_1$ for each of them. In each case this does not give the complete phase at the intersection point. Anyway, for each of these curves there are $R$ oriented curves that we can glue, and this can happen in two ways. Thus, we still get
$$2\times 2\left(\frac{\q{m_W}}{q-q^{-1}}\right)\times R=4\left(\frac{\q{m_W}}{q-q^{-1}}\right)R.$$
\end{itemize}
Finally, we recover the formula for $m'_\Gamma$.
\end{proof}

\bibliographystyle{plain}
\bibliography{biblio}

{\ncsc Institut de Math\'ematiques de Jussieu - Paris Rive Gauche\\[-15.5pt] 

Sorbonne Universit\'e\\[-15.5pt] 

4 place Jussieu,
75252 Paris Cedex 5,
France} \\[-15.5pt]

\medskip

{\it and}

\medskip

{\ncsc D\'epartement de math\'ematiques et applications\\[-15.5pt]

Ecole Normale Sup\'erieure\\[-15.5pt]

45 rue d'Ulm, 75230 Paris Cedex 5, France} \\[-15.5pt]

{\it E-mail address}: {\ntt thomas.blomme@imj-prg.fr}

\end{document}